\documentclass[11pt]{article}
\usepackage[utf8]{inputenc}
\usepackage{amssymb,amsmath, amsthm, mathabx}
\usepackage{color, enumerate}
\usepackage{comment}
\usepackage{hyperref}
\usepackage{a4wide}
\usepackage{accents}
\usepackage{tikz}
\usepackage{subfig}
\usepackage{longtable}

\usepackage[margin=1in]{geometry}
\usepackage{graphicx}
\usetikzlibrary{shapes}
\usetikzlibrary{plotmarks}
\usepackage{framed}
\usepackage{float}

\DeclareMathOperator{\bbE}{\mathbb{E}}

\DeclareMathOperator{\bbJ}{\mathbb{J}}

\DeclareMathOperator{\bbL}{\mathbb{L}}

\DeclareMathOperator{\bbN}{\mathbb{N}}

\DeclareMathOperator{\bbP}{\mathbb{P}}
\DeclareMathOperator{\bbQ}{\mathbb{Q}}
\DeclareMathOperator{\bbR}{\mathbb{R}}

\DeclareMathOperator{\bbZ}{\mathbb{Z}}

\DeclareMathOperator{\bfE}{\mathbf{E}}

\DeclareMathOperator{\bfI}{\mathbf{I}}
\DeclareMathOperator{\bfJ}{\mathbf{J}}

\DeclareMathOperator{\bfL}{\mathbf{L}}

\DeclareMathOperator{\bfP}{\mathbf{P}}

\DeclareMathOperator{\bfa}{\mathbf{a}}
\DeclareMathOperator{\bfb}{\mathbf{b}}

\DeclareMathOperator{\sM}{\mathcal{M}}

\DeclareMathOperator{\expm}{exp}

\DeclareMathOperator{\lc}{\lceil}
\DeclareMathOperator{\rc}{\rceil}
\DeclareMathOperator{\lf}{\lfloor}
\DeclareMathOperator{\rf}{\rfloor}
\DeclareMathOperator{\Var}{Var}

\theoremstyle{plain}
\newtheorem{thm}{Theorem}[section]
\newtheorem{prop}[thm]{Proposition}
\newtheorem{lem}[thm]{Lemma}
\newtheorem{cor}[thm]{Corollary}

\theoremstyle{definition}

\theoremstyle{remark}

\newtheorem{ex}[thm]{Example}

\DeclareMathOperator{\qmeas}{\bfP_{\bfa, \bfb}}
\DeclareMathOperator{\ameas}{\bbP}

\DeclareMathOperator{\aexp}{\bbE}
\DeclareMathOperator{\qexp}{\bfE_{\bfa,\bfb}}
\DeclareMathOperator{\qzexp}{\bfE_{\bfa,\bfb}^z}

\DeclareMathOperator{\E}{E}
\DeclareMathOperator{\En}{H}

\newcommand{\ubar}[1]{\underaccent{\bar}{#1}}

\DeclareMathOperator{\zs}{z_{\star}}
\DeclareMathOperator{\lam}{\lambda_{\star}}

\allowdisplaybreaks
\numberwithin{equation}{section}
\title{Large deviations for some corner growth models with inhomogeneity}

\author{Elnur Emrah\thanks{\footnotesize{Department of Mathematical Sciences, Carnegie Mellon University.\ Pittsburgh, PA, USA.\newline \href{mailto:eemrah@andrew.cmu.edu}{eemrah@andrew.cmu.edu}}}
\and
Chris Janjigian\thanks{Laboratoire de Probabilit\'es et Mod\`eles Al\'eatoires, Universit\'e Paris Diderot - Paris 7. Paris, France. \newline\href{mailto:cjanjigian@math.univ-paris-diderot.fr}{cjanjigian@math.univ-paris-diderot.fr}}\ \footnotemark[2]}

\begin{document}

\maketitle 
\abstract{We study an inhomogeneous generalization of the classical corner growth in which the weights are exponentially distributed with random parameters. Our interest is in the large deviation properties of the last passage times. We obtain tractable variational representations of the right tail large deviation rate functions in both the quenched and annealed settings and estimates for left tail large deviations. We also compute expansions of the right tail rate functions near the shape function, which are consistent with the expectation of KPZ type fluctuations in an appropriate regime.}

\noindent \textbf{Keywords:} corner growth model; 
directed last-passage percolation; TASEP; exactly solvable models; large deviations; rate functions
 
\noindent \textbf{AMS MSC 2010:} 60K35; 60K37.

\section{Introduction}
\label{S1}
Let $W = \{W(i, j): i, j \in \bbN\}$ be a collection of nonnegative real random variables (weights) with joint distribution $P$. We consider a growing random subset of $\bbN^2$.   Initially, this subset is empty and the growth rule is as follows: at time zero, we start a countdown of length $W(1,1)$; when this countdown ends, we add $(1,1)$ to the subset. This process then iterates: once the bottom and left neighbors (if they exist) of a site $(i,j)$ have been added, a countdown of length $W(i,j)$ begins; when it ends, $(i,j)$ joins the subset.

Our interest is in the last-passage times $G=\{G(i, j): i, j \in \bbN\}$, defined recursively by  
\begin{equation}
\begin{aligned}
\label{E57}
G(i, j) = G(i-1, j) \vee G(i, j-1) + W(i, j) \quad G(i, 0) = G(0, j) = 0 \qquad \text{ for } i, j \in \bbN. \\
\end{aligned}
\end{equation}
These random variables encode the evolution of the set described above in the sense that $(i, j) \in \bbN^2$ is added to the set at time $t = G(i, j)$. They can also be viewed as describing a directed last-passage percolation model since (\ref{E57}) is equivalent to 
\begin{equation}\label{E27}
G(m, n) = \max \limits_{\pi \in \Pi_{(1, 1), (m, n)}} \sum \limits_{(i, j) \in \pi} W(i, j),
\end{equation}
where $\Pi_{(k, l), (m, n)}$ is the set all directed paths from $(k, l)$ to $(m, n)$. A directed path $\pi \in \Pi_{(k, l), (m, n)}$ is a finite sequence $(u_i, v_i)_{i \in [p]}$ in $\bbZ^2$ such that $(u_1, v_1) = (k, l)$, $(u_p, v_p) = (m, n)$ and $(u_{i+1}-u_{i}, v_{i+1}-v_i) \in \{(1, 0), (0, 1)\}$ for $1 \le i < p$.

The corner growth model also maps to a generalization of the totally asymmetric exclusion process (TASEP) on $\bbZ$ started from step initial conditions. In this interpretation, we begin with particles at the sites $i < 0$ and holes at the sites $i \geq 0$. We label the particles with $i \in \bbN$, counting from right to left and holes with $j \in \bbN$ from left to right. In the dynamics, particle $i$ and hole $j$ interchange at time $G(i,j)$. If we denote the position of particle $i$ at time $t$ by $\sigma(i,t)$ then we have
\begin{align}
\sigma(i,t) = -i + \max\{j \in \bbN : G(i,j) \leq t\} \label{eq:particle}
\end{align} 
If the weights are i.i.d. with geometric or exponential marginals, this process is the usual TASEP run in discrete or continuous time, respectively \cite[p. 5]{Seppalainen}.

These and related models have received substantial research attention in the past two decades, partially in connection with KPZ universality. See the surveys \cite{Corwin,Martin}.  When $P$ is i.i.d. with geometric or exponential marginals, it has been possible to compute various statistics of the last-passage times. For example
\begin{equation}
\label{E58}
\lim \limits_{n \rightarrow \infty} \frac{G(\lf ns \rf, \lf nt \rf)}{n} = m(s+t) + 2\sigma \sqrt{st} \quad \text{ for } s, t > 0 \ P\text{-a.s.}, \end{equation}
where $m$ and $\sigma^2$ are the common mean and the variance of the weights. The exponential case of (\ref{E58}) was first proved in \cite{Rost} and the geometric case appeared in \cite{CohnElkiesPropp, JockuschProppShor, Seppalainen2}. Large deviation principles for the last-passage times were derived in \cite{Johansson, Seppalainen3}. These papers identified the right-tail rate function and the correct decay rate for both the right and left tails. It is also established in \cite{Johansson} that the model exhibits KPZ statistics; the fluctuations around the limit in (\ref{E58}) are of order $n^{1/3}$ and appropriately rescaled last-passage times converge to the Tracy-Widom GUE distribution. 

This paper concerns an inhomogeneous generalization of the classical i.i.d. exponential model. Given parameter sequences $\bfa = (a_n)_{n \in \bbN}$ and $\bfb = (b_n)_{n \in \bbN}$ taking values in $(0, \infty)$, we define a measure under which the weights are independent and $W(i, j)$ is exponentially distributed with mean $(a_i+b_j)^{-1}$. We state the assumptions of our model precisely in Section \ref{S3}, but for the moment we will outline our results in the case that $\bfa$ and $\bfb$ are independent i.i.d. sequences which are bounded away from zero and have finite means. We refer to the model where we condition on these random sequences as the quenched model. In the framework of particle systems, the quenched model corresponds to an inhomogeneous continuous time TASEP with particlewise and holewise inhomogeneity. The annealed measure is constructed by averaging the family of quenched measures over the joint distribution of $(\bfa,\bfb)$. Weights in the quenched model are independent, but not identically distributed in general; in the annealed model, the weights are identically distributed but correlated along rows and columns.

For the models we consider, the almost sure limit of $n^{-1} G(\lf ns \rf, \lf nt \rf)$ is deterministic and can be characterized as the solution of a certain variational problem, see \cite{Emrah}. With some choices of the marginal distributions of $\bfa$ and $\bfb$, this family of models has a feature not present in the previous exactly solvable corner growth models: the presence of linear segments of the shape function. This is of some interest because different fluctuation exponents are expected in these different regions if the weights are weakly correlated. For a particle system perspective on this phenomenon in the case where $\bfa$ is almost surely constant and $\bfb$ is i.i.d. and bounded, see \cite{KrugSepp}.

The present paper is devoted to the question of large deviations corresponding (\ref{E58}). In the quenched setting, we are able to prove a large deviation principle with rate $n$ and a rate function given by the solution to a reasonably tractable variational problem. With certain choices of the weights and in certain directions, we provide some explicit formulas for these rate functions.

In the annealed setting, we have a variational expression for the right tail rate function which is similar to the variational expression in the quenched setting, though we no longer have any non-trivial explicitly computable examples. Deviations to the right of the shape function in the annealed model are connected to deviations in the quenched model through a variational problem involving relative entropy. Heuristically, these deviations should arise from perturbations of $(\bfa,\bfb)$ combined with deviations in the quenched model with these perturbed parameters. There are rate $n$ annealed large deviations to the left of the shape function. This is in contrast to the i.i.d. models, where the rate is $n^2$ \cite{Johansson, Seppalainen3}. We show that this occurs by using the fact that it is possible to see a finite entropy deviation of the (order $n$ many) parameters $\{a_i\}_{i=1}^{\lf ns \rf}$ and $\{b_j\}_{j=1}^{\lf nt \rf}$ which affect the distribution of $G(\lf ns \rf, \lf nt \rf)$ and makes the shape function smaller.

We identify the expansions of both the quenched and annealed rate functions near the shape function. In the quenched model, for directions in which the shape function is strictly concave, these expansions are heuristically consistent with the expectation of Tracy-Widom GUE fluctuations. Fluctuation results for an inhomogeneous version of the closely related Sepp\"al\"ainen-Johansson model (oriented digital boiling) were previously obtained in a series of papers by Gravner, Tracy, and Widom \cite{GTW01,GTW02b,GTW}. The question of what happens at the interface of the linear and concave regions was left open in those papers. At the interface of the linear and concave regions for the models we study, our results suggest KPZ type fluctuations under a moment condition. We elaborate on this connection in the next section.

To prove the variational formulas for the right tail rate functions, we follow an approach introduced in \cite{Seppalainen4} and applied in \cite{GeorgSepp, Janjigian, Seppalainen3}. The key technical condition making this scheme tractable is an analogue of Burke's theorem from queueing theory, which in this setting corresponds to the existence of a stationary version of the model, as discussed in Proposition \ref{P2}. We expect that the techniques employed in this paper could be used to obtain similar results in inhomogeneous versions of other models with the Burke property, such as the log gamma polymer \cite{Sep12}, the strict-weak polymer \cite{CSS14} and the corner growth model with geometric weights \cite{Emrah}.

The principal contributions of this paper are as follows. To the best of our knowledge, this is the first inhomogeneous model in the KPZ class for which exact large deviation rate functions have been computed. The rate functions we obtain in both the quenched and annealed settings are tractable. For general choices of the distributions of the sequences $(\bfa,\bfb)$, we identify the asymptotic rate that the right tail rate function tends to zero near the shape function, suggesting KPZ type fluctuations for the quenched model in appropriate directions. In particular, our results suggest a partial answer to the problem of what type of fluctuations to expect at the interface of the linear and concave regions. We further connect our quenched and annealed rate functions through a natural variational problem involving relative entropy.

The paper is organized as follows. In Section \ref{S3}, we define the model precisely and state our results. The remaining sections are devoted to proofs. In Section \ref{S2}, we discuss the stationary model and compute the Lyapunov exponents of the last passage times.  In Section \ref{Sreg}, we study the extremizers of the variational problems for the rate functions and Lyapunov exponents. We then estimate the probability of left tail large deviations in Section \ref{S4}. In Section \ref{S7} we show that the Legendre-Fenchel transform of the right tail rate function is given by the previously computed Lyapunov exponents. These results are combined to prove the large deviation principle for the quenched model.  In Section \ref{S8}, we note that the extremizers for the annealed model are connected to the extremizers for a quenched model with different parameters, which gives a variational connection between the quenched and annealed rate functions. Understanding of the extremizers also allows us to prove the scaling estimates in Section \ref{S10}. We include the standard subadditivity arguments showing existence and regularity of the Lyapunov exponents and right tail rate functions in Appendix \ref{A}.

\noindent \textbf{Notation}.
For real numbers $a,b$, we denote $\max(a,b) = a\vee b$ and $\min(a,b) = a \wedge b$. We take the convention that $\bbN = \{n \in \bbZ: n > 0\}$ and $\bbR_+ = \{x \in \bbR : x > 0\}$. For $D \subset \bbR$, we denote by $\sM_1(D)$ the collection of probability measures on $D$. For $\eta \in \sM_1(D)$, we use the notation $\ubar{\eta} = \text{ess-inf}\{\eta\}$ and $\bar{\eta} = \text{ess-sup}\{\eta\}$.

Given $\nu,\mu \in \sM_1(D)$, the relative entropy is defined by $\En (\nu|\mu) = \E^\nu \log \frac{d\nu}{d\mu}$ if $\nu$ is absolutely continuous with respect to $\mu$ and $\infty$ otherwise. See for example the discussions in \cite{DemboZeitouni} and \cite{Rassoul-AghaSeppalainen} for basic properties of the relative entropy. We denote absolute continuity of $\nu$ with respect to $\mu$ by $\nu \ll \mu$. For probability measures $\nu,\mu$ with $\nu \ll \mu$, we write $\frac{d\nu}{d\mu}(x) \simeq f(x)$ if $\nu(dx) = [\int f(x) \mu(dx)]^{-1} f(x) \mu(dx)$. For a probability measure $\mu$ on $\bbR_+$, define
\begin{align}\label{Mmu}
\sM^\mu &= \{\nu \in \sM_1(\bbR_+) : \En (\nu | \mu) < \infty\}
\end{align}
and note that for each $\mu$, $\sM^\mu$ is a convex set by convexity of $\En (\nu|\mu)$.

For $f:\mathbb{R} \to (-\infty,\infty]$, $f^{\star}(\xi) = \sup_{x \in \mathbb{R}}\{x \xi - f(x)\}$ defines the Legendre-Fenchel transform. We refer the reader to \cite{Rockafellar} for basic properties of this transform.

\noindent \textbf{Acknowledgements} The authors are grateful to Timo Sepp\"al\"ainen for many helpful conversations. The authors would also like to thank the Department of Mathematics at the University of Wisconsin--Madison, where much of this work was completed.

\section{Model and results}
\label{S3}
\subsection{Model}
Denote by $W(i, j)$ the projection $\bbR_+^{\bbN^2} \rightarrow \bbR_+$ onto the coordinate $(i, j)$ for $i, j \in \bbN$. For any sequences $\bfa=(a_1,a_2,\dots),\bfb=(b_1,b_2,\dots)$ taking values in $\bbR_+$, we define $\bfP_{\bfa, \bfb}$ to be the product measure on $\bbR_+^{\bbN^2}$ satisfying 
\begin{align*}
\bfP_{\bfa, \bfb}(W(i, j) \ge x) = e^{-(a_i+b_j)x} \quad \text{for } i, j \in \bbN \text{ and } x \ge 0.
\end{align*}
We will draw the sequences $(\bfa,\bfb)$ randomly from a distribution $\mu$ on $\bbR_+^{\bbN} \times \bbR_+^{\bbN}$. For $k \in \bbZ_+$, let $\tau_k$ denote the shift $(c_n)_{n \in \bbN} \mapsto (c_{n+k})_{n \in \bbN}$. In all of the results that follow, we make the following assumptions on $(\bfa,\bfb)$. We assume that $\bfa$ and $\bfb$ are stationary sequences under $\mu$. We assume further that $\mu$ is separately ergodic with respect to $\tau_k \times \tau_l$ for $k, l \in \bbN$. This means that if $k, l \in \bbN$ and $B \subset \bbR_+^{\bbN} \times \bbR_+^{\bbN}$ is a Borel set with $(\tau_k \times \tau_l)^{-1}(B) = B$ then $\mu(B) \in \{0, 1\}$.  

The annealed distribution $\bbP$ is given by $\ameas(B) = \E \left[\qmeas(B)\right]$ for any Borel set $B \subset \bbR_+^{\bbN^2}$, where $\E$ is the expectation under $\mu$. Let $\bfE_{\bfa, \bfb}$ and $\bbE$ denote the expectations under $\bfP_{\bfa, \bfb}$ and $\bbP$, respectively. We denote by $\alpha$ and $\beta$ the distributions of $a_1$ and $b_1$ and take the convention that $a$ and $b$ are random variables with distributions $\alpha$ and $\beta$ respectively. In all of the following results, we will assume that $\E[a+b] < \infty$ and $\ubar{\alpha}+\ubar{\beta} > 0$. Finally, all large deviation results under $\ameas$ are limited to the case in which $\bfa$ and $\bfb$ are independent i.i.d. sequences.

We will also consider a `stationary' model defined on the extended sample space $\bbR_+^{\bbZ_+^2}$. Each weight $W(i, j)$ is now redefined as the projection onto coordinate $(i, j)$ for $i, j \in \bbZ_+^2$. Introduce the last-passage times 
\begin{align}
\label{E52}
\widehat{G}(m, n) = \max_{\pi \in \Pi_{(0, 0), (m, n)}} \sum_{i, j \in \pi} W(i, j) \quad \text{ for } m, n \in \bbZ_+.\end{align} For sequences $\bfa$ and $\bfb$ in $(0, \infty)$ and $z \in (-\ubar{\alpha}, \ubar{\beta})$, define the product measure $\bfP_{\bfa, \bfb}^z$ on $\bbR_+^{\bbZ_+^2}$ by 
\begin{equation}
\label{E30}
\begin{aligned}
&\bfP_{\bfa, \bfb}^z(W(i, j) \ge x) = \exp(-(a_i+b_j)x) \qquad \bfP_{\bfa, \bfb}^z(W(0, 0) = 0) = 1 \\
&\bfP_{\bfa, \bfb}^z(W(i, 0) \ge x) = \exp(-(a_i+z)x) \qquad \ \bfP_{\bfa, \bfb}^z(W(0, j) \ge x) = \exp(-(b_j-z)x)
\end{aligned}
\end{equation}
for $i, j \in \bbZ_+$ and $x \ge 0$. We will use definition (\ref{E30}) for $z = -\ubar{\alpha}$ when $a_i > \ubar{\alpha}$ for $i \in \bbN$ and for $z = \ubar{\beta}$ when $b_j > \ubar{\beta}$ for $j \in \bbN$. The utility of these measures is that the last-passage increments given by $I(m, n) = \widehat{G}(m, n)-\widehat{G}(m-1, n)$ for $m \ge 1, n \ge 0$ and
$J(m, n) = \widehat{G}(m, n)-\widehat{G}(m, n-1)$ for $m \ge 0, n \ge 1$ are stationary in the following sense. 
\begin{prop}[Proposition 4.1 in \cite{Emrah}]
\label{P2}
Let $k, l \in \bbZ_+$. Under $\bfP_{\bfa, \bfb}^z$,  
\begin{enumerate}[(a)]
\item $I(i, l)$ has the same distribution as $W(i, 0)$ for $i \in \bbN$. 
\item $J(k, j)$ has the same distribution as $W(0, j)$ for $j \in \bbN$. 
\item The random variables $\{I(i, l): i > k\} \cup \{J(k, j): j > l\}$ are jointly independent. 
\end{enumerate}
\end{prop}

For admissible $z$, define the measure $\bbP^z$ on $\bbR_+^{\bbZ_+^2}$ by 
$\bbP^z(B) = \E [\bfP_{\bfa, \bfb}^z(B)]$
for any Borel set $B$. Let $\bfE_{\bfa, \bfb}^z$ and $\bbE^z$ denote the expectations under $\bfP_{\bfa, \bfb}^z$ and $\bbP^z$, respectively. 

\subsection{Results}

We begin by briefly summarizing the results from \cite{Emrah}. The ergodicity assumptions on $\mu$ and the superadditivity of the last-passage times imply that $\lim_{n \rightarrow \infty} n^{-1}G(\lf ns \rf, \lf nt \rf) = g(s, t)$ for $s, t > 0$ $\bbP$-a.s. and $\qmeas$-a.s. for $\mu$-a.e. $(\bfa,\bfb)$ for some deterministic function $g$ known as the \emph{shape function}. $g$ admits the variational representation 
\begin{equation}
\label{E28}
g(s, t) = \inf \limits_{z \in [-\ubar{\alpha}, \ubar{\beta}]} \left\{s \E \left[\frac{1}{a+z}\right] + t \E \left[\frac{1}{b-z}\right]\right\} \quad \text{for } s, t > 0.
\end{equation}
The infimum above is actually a minimum with a unique minimizer and the function $g_z$ given by $g_z(s,t) = s \E \left[(a+z)^{-1}\right] + t \E \left[(b-z)^{-1}\right]$ is the shape function in the stationary version of the model. At times we will also view $g(s,t)$ as a function of $(\alpha, \beta) \in \sM_1(\bbR_+)^2$. In these cases, we will use the notation  $(\alpha,\beta) \mapsto g(s,t) \equiv g_{\alpha,\beta}(s,t)$ to highlight the dependence on these measures. This map will be considered for any $(\alpha,\beta) \in \sM_1(\bbR_+)^2$.

Set \begin{equation}
\label{E67}
\begin{aligned}
c_1 = \frac{\E\left[(b+\ubar{\alpha})^{-2}\right]}{\E\left[(a-\ubar{\alpha})^{-2}\right]} \qquad  c_2 = \frac{\E\left[(b-\ubar{\beta})^{-2}\right]}{\E\left[(a+\ubar{\beta})^{-2}\right]}.
\end{aligned}
\end{equation}
Then $0 \le c_1 < c_2 \le \infty$, and $c_1 = 0$ and $c_2 = \infty$ if and only if $\E[(a-\ubar{\alpha})^{-2}] = \infty$ and $\E[(b-\ubar{\beta})^{-2}] = \infty$, respectively. It can be seen from (\ref{E28}) that $g$ is strictly concave for $c_1 < s/t < c_2$ and is linear for $s/t \le c_1$ or $s/t \ge c_2$, see Figure \ref{F1}. 
\begin{figure}[h!]
\centering
\begin{tikzpicture}[scale = 1]
\draw[<->](-1, 5.5)node[above]{$t$}--(-1, -1)--(5.5, -1)node[right]{$s$};
\draw (-1, -1)node[below]{$0$};
\draw[blue] (-1, -1) -- (1, 5); 
\draw[blue] (-1, -1) -- (5, 2); 
\fill[red, fill opacity=0.3](-1, -1)--(0, 2)--(-1, 5)--cycle;
\fill[red, fill opacity=0.3](-1, -1)--(4, -1)--(1.55, 0.275)--cycle;
\fill[red, fill opacity=0.3](-1, -1)--(0, 2)--(1.55, 0.275)--cycle;
\fill [white, very thick](0,2) arc (-160:-123:4 and 3.5);
\draw [red, very thick] (0, 2) -- (-1, 5); 
\draw [red, very thick] (1.55, 0.275) -- (4, -1); 
\draw [red, very thick](0,2) arc (-160:-123:4 and 3.5); 
\draw [red, very thick](1.3, 1.3)node[below]{$g \le 1$};
\draw [blue, very thick](1, 5)node[above]{\large $s/t = c_1$};
\draw [blue, very thick](5, 2)node[above]{\large $s/t = c_2$};
\end{tikzpicture}
\caption{\small{An illustration of the sublevel set $g \le 1$ and the rays $s/t = c_1$ and $s/t = c_2$ when $0 < c_1 < c_2 < \infty$.}}
\label{F1}
\end{figure}
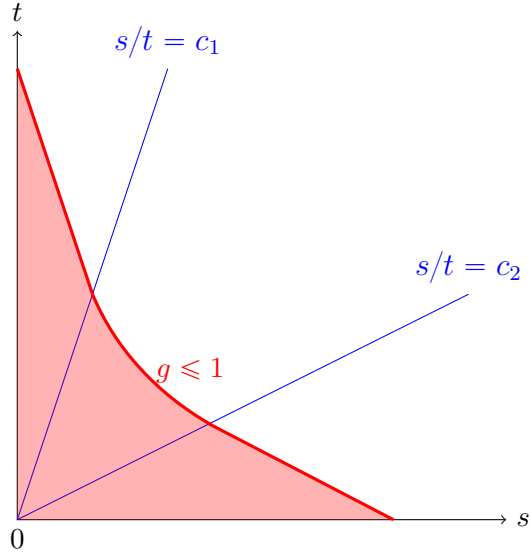

We show in Proposition \ref{P1} of the appendix that for $s, t, \lambda > 0$,  we may define the quenched and annealed Lyapunov exponents by
\begin{align*}
\bfL_{s, t}(\lambda) = \lim_{n \rightarrow \infty} \frac{1}{n} \log \bfE_{\bfa, \bfb}\left[e^{\lambda G(\lf ns \rf, \lf nt \rf)}\right] \quad \mu\text{-a.s.}, \qquad \bbL_{s, t}(\lambda) =  \lim_{n \rightarrow \infty} \frac{1}{n} \log \bbE \left[e^{\lambda G(\lf ns \rf, \lf nt \rf)}\right].  \end{align*}
Our first result is an exact computation of these exponents.
\begin{thm}
\label{T1} For $s,t,\lambda > 0$,
\begin{align}
\bfL_{s, t}(\lambda) &=  \begin{cases} 
\displaystyle \inf \limits_{z \in [-\ubar{\alpha}, \ubar{\beta}-\lambda]} \left\{s \E \log\left(\frac{a+z+\lambda}{a+z}\right) + t \E \log\left(\frac{b-z}{b-z-\lambda}\right)\right\}&\text{ if } 0 < \lambda \le  \ubar{\alpha} + \ubar{\beta} \\ \infty &\text{ if } \lambda > \ubar{\alpha}+\ubar{\beta}.\end{cases}\label{E146}\\
\bbL_{s, t}(\lambda) &=  \begin{cases}
\displaystyle \inf \limits_{z \in [-\ubar{\alpha}, \ubar{\beta}-\lambda]} \left\{s \log \E \left(\frac{a+z+\lambda}{a+z}\right) + t \log \E \left(\frac{b-z}{b-z-\lambda}\right)\right\}&\text{ if } 0 < \lambda \le \ubar{\alpha} + \ubar{\beta} \\ \infty &\text{ if } \lambda > \ubar{\alpha}+\ubar{\beta}\end{cases}\label{E147}
\end{align}
\end{thm}

Having proven Theorem \ref{T1}, a proof similar to the proof of Theorem \ref{T1} allows us to compute the Lyapunov exponents in the stationary version of the model.
\begin{thm}\label{thm:statLyapunov}
For $z \in (-\ubar{\alpha},\ubar{\beta})$, almost surely for all $s,t>0$ and $\lambda \in (0, (\ubar{\alpha} + z) \wedge (\ubar{\beta} - z))$
\begin{align*}
&\bfL_{s,t}^z (\lambda) := \lim_{n\to\infty}n^{-1}\log \qzexp\left[e^{\lambda \hat{G}(\lf ns \rf, \lf nt \rf)}\right]\\
=&\left\{s \E\left[ \log \frac{a+z}{a+z-\lambda}\right] + t \E\left[ \log \frac{b-z+\lambda}{b-z}\right]\right\} \vee\left\{s \E \left[ \log \frac{a+z+\lambda}{a+z}\right] + t \E\left[ \log \frac{b-z}{b-z-\lambda}\right]\right\}.
\end{align*}
\end{thm}

Similarly, we show in Proposition \ref{P3} that for $s,t > 0$ and $r \in \bbR$, we may define right tail rate functions by
\begin{align*}
&\lim \limits_{n \rightarrow \infty} -\frac{1}{n} \log \bfP_{\bfa, \bfb}(G(\lf ns \rf, \lf nt \rf) \ge nr) = \bfJ_{s, t}(r) \quad \mu\text{-a.s.}, \\
&\lim \limits_{n \rightarrow \infty} -\frac{1}{n} \log \bbP(G(\lf ns \rf, \lf nt \rf) \ge nr) = \bbJ_{s, t}(r) 
\end{align*}

Using Theorem \ref{T1}, we show that
\begin{thm}
\label{thm:Jcd} For $s,t > 0$,
\begin{align}
\bfJ_{s,t}(r) &= \begin{cases}\displaystyle \sup \limits_{\substack{\lambda \in (0,\ubar{\alpha} + \ubar{\beta}] \\ z \in [-\ubar{\alpha}, \ubar{\beta} - \lambda]}} \left\{ r \lambda - s \E\log\left(\frac{a+z+\lambda}{a+z}\right) - t \E\log\left(\frac{b-z}{b-z-\lambda}\right) \right\} & r \geq g(s,t) \label{E26}\\
0 & r < g(s,t) \end{cases}\\
\bbJ_{s,t}(r) &= \begin{cases}
\displaystyle\sup \limits_{\substack{\lambda \in (0,\ubar{\alpha} + \ubar{\beta}] \\ z \in [-\ubar{\alpha}, \ubar{\beta} - \lambda]}} \left\{ r \lambda - s \log\E\left[\frac{a+z+\lambda}{a+z}\right] - t \log\E\left[\frac{b-z}{b-z-\lambda}\right] \right\} & r \geq g(s,t) \label{E26.1} \\
0 & r < g(s,t)
\end{cases}
\end{align}
\end{thm}

The preceding result also describes left tail large deviations for a tagged particle in an inhomogeneous TASEP with step initial condition. This TASEP can be obtained from the corner growth by defining the position of particle $i \in \bbN$ at time $t \ge 0$ as in (\ref{eq:particle}). By monotonicity of $G$, the particles remain ordered i.e. $\sigma(i, t) > \sigma(i+1, t)$ for $i \in \bbN$ and $t \ge 0$.  Initially, $\sigma(i, 0) = -i$ for $i \in \bbN$ and particles move on $\bbZ$ over time according to the following rule. If particle $i$ is at site $-i+j-1$, as soon as site $-i+j$ is vacant, particle $i$ moves to site $-i+j$ after $W(i, j)$ amount of time. Since $\{\sigma(i, t) > j\} = \{G(i, i+j) < t\}$ as events, Theorem \ref{thm:Jcd} implies the next corollary. 
\begin{cor}
For $x, y, t > 0$, 
\begin{align*}
&\lim_{n \rightarrow \infty}-\frac{1}{n} \log \bfP_{\bfa, \bfb}(\sigma(\lf nx \rf, nt) > \lf ny \rf) = \bfJ_{x, x+y}(t) \quad \text{ a.s. }\\
&\lim_{n \rightarrow \infty}-\frac{1}{n} \log \bbP(\sigma(\lf nx \rf, nt) > \lf ny \rf) = \bbJ_{x, x+y}(t). 
\end{align*}
\end{cor}

As with the shape function, we will at times consider the maps $(\alpha,\beta) \mapsto \bfJ_{s,t}(r) \equiv \bfJ_{s,t}^{\alpha,\beta}(r)$ and $(\alpha,\beta) \mapsto \bbJ_{s,t}(r) \equiv \bbJ_{s,t}^{\alpha,\beta}(r)$.

Note that the Lyapunov exponents and the right tail rate functions depend on $\mu$ only through the marginal distributions $\alpha$ and $\beta$. The variational problem in (\ref{E26}) can be solved exactly for certain choices of $\alpha,\beta, s$ and $t$. We note that if $r \geq g(s,t)$ and there exists $\lam \in (0, \ubar{\alpha}+\ubar{\beta})$ and $\zs \in (-\ubar{\alpha}, \ubar{\beta}-\lam)$ such that 
\begin{align*}
0 &= s\E\left[\frac{1}{a+\zs+\lam}-\frac{1}{a+\zs}\right] + t \E\left[\frac{1}{b-\zs-\lam}-\frac{1}{b-\zs}\right] \\
r &= s\E \left[\dfrac{1}{a+\zs+\lam}\right] + t \E \left[\dfrac{1}{b-\zs-\lam}\right], 
\end{align*}
then 
\begin{align}
\label{E122}
\bfJ_{s, t}(r) = \lam r - s\E\log\left(\frac{a+\zs+\lam}{a+\zs}\right) + t\E\log\left(\frac{b-\zs}{b-\zs-\lam}\right). 
\end{align}

\begin{ex} If $\alpha = \beta = \delta_{c/2}$ for $c > 0$, then for $r \geq g(s,t) = c^{-1}(\sqrt{s} + \sqrt{t})^2$,
\begin{equation}
\label{E65}
\bfJ_{s, t}(r) = \sqrt{(s+t-cr)^2-4st} -2s \cosh^{-1}\left(\frac{s-t+cr}{2\sqrt{csr}}\right) - 2t \cosh^{-1}\left(\frac{t-s+cr}{2\sqrt{ctr}}\right),
\end{equation}
which recovers \cite[Theorem 4.4]{Seppalainen3}.
\end{ex}
\begin{ex}
If $\alpha = \beta = p\delta_{c} + q\delta_{d}$ for $p, q, c, d > 0$ with $p+q = 1$ and $s = t$, then for $r \geq g(s,s) = 2s\left(pc^{-1} + qd^{-1}\right),$
\begin{align*}
\bfJ_{s, s}(r) &= r\lam - sp \log\left(\frac{c+\zs+\lam}{c+\zs}\right)- tq \log\left(\frac{c-\zs}{c-\zs-\lam}\right)\\
&-sq \log\left(\frac{d+\zs+\lam}{d+\zs}\right)- tq \log\left(\frac{d-\zs}{d-\zs-\lam}\right)
\end{align*}
where 
\begin{align*}
\zs &= \frac{2cp + 2dq + c^2r + d^2r -\sqrt{\Delta}}{2r}, \qquad \zs+\lam = \frac{2cp + 2dq + c^2r + d^2r + \sqrt{\Delta}}{2r},\\
\Delta &= (2 c p+ 2 d q + c^2 r + d^2 r)^2 + 4 r (2 c d^2 p +2 c^2 d q - c^2 d^2 r). 
\end{align*}
More complicated exact formulas in this model are available in all directions $(s,t)$.
\end{ex}
\begin{ex}
If $\alpha$ and $\beta$ are uniform on $[c/2, c/2+l]$ for $c, l > 0$ and $s = t$, then
\begin{align*}
\bfJ_{s, s}(r) = r\lam -\frac{2s}{l} \int_{c/2}^{c/2+l} \log\left(\frac{x+\zs+\lam}{x+\zs}\right)dx \quad \text{ for } r \ge g(s, s) = \frac{2s}{l}\log\left(1+\frac{2l}{c}\right),
\end{align*}
where
\begin{align*}
\zs = - \sqrt{\frac{(c/2+l)^2-c^2e^{rl/s}/4}{1-e^{rl/s}}} \qquad \zs+\lam = \sqrt{\frac{(c/2+l)^2-c^2e^{rl/s}/4}{1-e^{rl/s}}}.
\end{align*}
\end{ex}

Left tail large deviations in the quenched model have rate strictly larger than $n$. We expect that under mild hypotheses the correct rate should be $n^2$, as is the case in the homogeneous model where $\alpha = \beta = \delta_{\frac{c}{2}}$ \cite{Johansson, Seppalainen3}.
\begin{lem} \label{quenchedlb}
$\lim \limits_{n \to \infty} -\dfrac{1}{n}\log  \qmeas \left(G(\lfloor n s \rfloor, \lfloor n t\rfloor ) \leq nr \right) = \infty$ for $s, t > 0$ and $r < g(s, t)$ $\mu$-a.s.
\end{lem}
Combining our results for the right and left tail deviations, we can prove a full quenched LDP at rate $n$. The rate function is given by
\begin{align} \label{Istdef}
\bfI_{s,t}(r) &= \begin{cases}
\bfJ_{s,t}(r) & r \geq g(s,t)\\
\infty & r < g(s,t)
\end{cases}.
\end{align}
As before, we will at times use the notation $(\alpha,\beta) \mapsto \bfI_{s,t}(r) \equiv \bfI_{s,t}^{\alpha,\beta}(r)$.
\begin{thm}
\label{T0}
$\mu$-a.s, for any $s, t >0$, the distribution of $n^{-1}G(\lf ns \rf, \lf nt \rf)$ under $\bfP_{\bfa, \bfb}$ satisfies a large deviation principle with rate $n$ and convex, good rate function $\bfI_{s,t}$.
\end{thm}
Although our proof of the large deviation principle goes through the Lyapunov exponents, we do not apply the G\"artner-Ellis theorem. The steepness condition in this model is $E[(a-\ubar{\alpha})^{-1}] = E[(b-\ubar{\beta})^{-1}]= \infty$, which would rule out having linear segments of the shape function and so is too restrictive.

In contrast to the quenched case, there are non-trivial annealed large deviations at rate $n$. The following bound gives a mechanism for these deviations.
\begin{lem}\label{annealedub}
For any $x < y$, 
\begin{align*}
\limsup_{n \to \infty} - \frac{1}{n} \log \ameas(n^{-1}G(\lf ns \rf, \lf nt \rf) \in (x,y)) \leq \inf_{\substack{\nu_1 \in \sM^\alpha, \nu_2 \in \sM^\beta \\ g_{\nu_1,\nu_2}(s,t) \in (x,y)}} \{s \En (\nu_1 | \alpha) + t \En (\nu_2 | \beta)\}
\end{align*}
\end{lem}
The other bound needed to show that $n$ is the correct rate for certain left tail large deviations follows from essentially the same argument used to show that the quenched rate is strictly larger than $n$. This is discussed briefly after the proof of Lemma \ref{quenchedlb}. To show that there are rate $n$ annealed left tail large deviations it suffices to show that there exist $\nu_1 \in \sM^\alpha$ and $\nu_2\in \sM^\beta$ with $g_{\nu_1,\nu_2}(s,t) < g_{\alpha,\beta}(s,t)$. We give a simple proof that under mild conditions this is the case in Lemma \ref{sufconlower}. We expect that this mechanism is not sharp. 

\begin{ex} \label{notsharp}
Suppose that $\alpha = \frac{1}{2}\delta_1 + \frac{1}{2} \delta_2$ and $\beta = \delta_1$, and recall that $\sM^\alpha = \{p \delta_1 + (1-p) \delta_2: 0\leq p \leq 1\}$. For $0 \leq p \leq 1$, call $\alpha_p = p \delta_1 + (1-p) \delta_2$. Then $\{g_{\alpha_p,\beta}(1,9) : 0 \leq p \leq 1\}=\{5.\bar{3}\} \cup (5.5,8]$. The reason for the discontinuity in this example is that if $p > 0$, then the functional in (\ref{E28}) is minimized on the set $(-1, 1)$, but if $p = 0$, the minimization occurs on $(-2, 1)$. We have chosen $s=1,t=9$ so that the minimizer for the $p=0$ case occurs in $(-2 , -1)$. The bound one obtains from Lemma \ref{annealedub} in this example is infinite when applied to the interval $(5.4,5.5)$. The finite relative entropy perturbation of the $a_i$ parameters switching the distribution to $\delta_2$ turns this into a right tail large deviation.
\end{ex}

The next theorem connects quenched rate function and annealed right tail rate function through a variational problem. We expect that this result means that large deviations above the shape function in the annealed model with marginals $\alpha$ and $\beta$ can be viewed as a large deviation in the parameters $\{a_i\}_{i=1}^{\lf ns \rf}$ and $\{b_j\}_{j=1}^{\lf nt \rf}$ which affect the distribution of $G(\lf ns \rf, \lf nt \rf)$, followed by a deviation in the quenched model with these perturbed parameters. Our proof is purely analytic and does not show this interpretation directly. A similar, but stronger, connection was shown for random walk in a random environment by Comets, Gantert and Zeitouni in \cite{CometsGantertZeitouni}.
\begin{thm} \label{thm:annealedent}
For any $s,t > 0$ and $r > g(s,t)$,
\begin{align*}
\bbJ_{s,t}^{\alpha,\beta}(r) &=  \inf \limits_{\substack{\nu_1 \in \sM^\alpha \\ \nu_2 \in \sM^\beta}} \left\{\bfI_{s,t}^{\nu_1,\nu_2}(r) + s \En (\nu_1 | \alpha) + t \En (\nu_2 | \beta)\right\}.
\end{align*}
A minimizing pair $(\nu_1,\nu_2)$ exists and the equality
\begin{align*}
\bbJ_{s,t}^{\alpha,\beta}(r) &=  \bfI_{s,t}^{\nu_1,\nu_2}(r) + s \En (\nu_1 | \alpha) + t \En (\nu_2 | \beta)
\end{align*}
holds if and only if
\begin{align*}
\frac{d\nu_1}{d\alpha}(a) \simeq \frac{a+\zs+\lam}{a+\zs}, \qquad \frac{d\nu_2}{d\beta}(b) \simeq \frac{b-\zs}{b-\zs-\lam}
\end{align*}
where $\zs$ and $\lam$ are the unique $\zs,\lam$ with $\lam \in [0,\ubar{\alpha} + \ubar{\beta}], \zs \in [-\ubar{\alpha}, \ubar{\beta} - \lam]$ satisfying
\begin{align}\label{maxpair}
\bbJ_{s,t}^{\alpha,\beta}(r) = r \lam - s \log \E^\alpha\left[\frac{a+\zs+\lam}{a+\zs}\right] - t \log \E^\beta\left[\frac{b-\zs}{b-\zs-\lam}\right].
\end{align}
\end{thm}
\noindent It is natural to conjecture that this variational connection describes all rate $n$ annealed large deviations, rather than just annealed right tail large deviations. We have been unable to prove this result.

The next result concerns the regularity of our rate functions. Our rate functions are convex and differentiable to the right of $g(s,t)$, but we note that for certain choices of $\alpha$ and $\beta$ they can have linear segments; see Lemma \ref{L1} and the comments preceding it. 
\begin{thm}
\label{T4}
For any $s, t > 0$, both $\bfJ_{s, t}$ and $\bbJ_{s, t}$ are continuously differentiable on $[g(s, t), +\infty)$. 
\end{thm}

Finally, we describe the leading order asymptotics of $\bfJ_{s, t}(r)$ and $\bbJ_{s, t}(r)$ as $r \downarrow g(s, t)$ and comment on the implications for the fluctuations of the last-passage times. Let $\zeta$ denote the unique minimizer of (\ref{E28}).
 
\begin{thm}
\label{T3}
For any $s, t > 0$, as $\epsilon \downarrow 0$, 
\[
\bfJ_{s, t}(g(s, t) + \epsilon) = 
\begin{cases}
\left(-s\E\left[\dfrac{2}{(a-\ubar{\alpha})^2}\right] + t\E\left[\dfrac{2}{(b+\ubar{\alpha})^2}\right]\right)^{-1} \epsilon^{2} + o(\epsilon^{2}) &\text{ if } s/t < c_1 \\
\dfrac{2}{3}\left(s\E\left[\dfrac{1}{(a-\ubar{\alpha})^3}\right] + t\E\left[\dfrac{1}{(b+\ubar{\alpha})^3}\right]\right)^{-1/2}\epsilon^{3/2} + o(\epsilon^{3/2}) &\text{ if } s/t = c_1 \\ 
&\text{ and } \E[(a-\ubar{\alpha})^{-3}] < \infty \\
\dfrac{4}{3}\left(s\E\left[\dfrac{1}{(a+\zeta)^3}\right] + t\E\left[\dfrac{1}{(b-\zeta)^3}\right]\right)^{-1/2}\epsilon^{3/2} + o(\epsilon^{3/2}) &\text{ if } c_1 < s/t < c_2 \\ 
\dfrac{2}{3}\left(s\E\left[\dfrac{1}{(a+\ubar{\beta})^3}\right] + t\E\left[\dfrac{1}{(b-\beta)^3}\right]\right)^{-1/2}\epsilon^{3/2} + o(\epsilon^{3/2}) &\text{ if } s/t = c_2 \\ 
&\text{ and } \E[(b-\ubar{\beta})^{-3}] < \infty\\
\left(s\E\left[\dfrac{2}{(a+\ubar{\beta})^2}\right] - t\E\left[\dfrac{2}{(b-\ubar{\beta})^2}\right]\right)^{-1} \epsilon^{2} + o(\epsilon^{2}) &\text{ if } s/t > c_2 \\
\end{cases}
\]
\end{thm}
\noindent We do not have an intuitive explanation for the presence of an extra factor of $\frac{1}{2}$ in the boundary cases $\frac{s}{t}=c_1,c_2$. 

The results of Theorem \ref{T3} in the concave region $S$ and the boundary lines $\frac{s}{t} = c_1$ or $c_2$ are heuristically consistent with the expectation of KPZ type fluctuations. For example, to see this set
\begin{align*}
C &= s\E\left[\dfrac{1}{(a+\zeta)^3}\right] + t\E\left[\dfrac{1}{(b-\zeta)^3}\right] = \frac{1}{2} \partial_z^2 g_z(s,t) \big|_{z = \zeta}
\end{align*}
and assume that our asymptotic result in the concave region hold for finite $n$. Then for $(s,t) \in S$ and large $r$, we expect to see
\begin{align*}
\qmeas(G(\lf ns \rf, \lf nt \rf) - n g(s,t) \geq n^\frac{1}{3} C^\frac{1}{3} r) \approx \text{exp}\left\{-\frac{4}{3}C^{-\frac{1}{2}} (C^\frac{1}{3} n^{-\frac{2}{3}} r)^{\frac{3}{2}}n\right\} = e^{-\frac{4}{3} r^\frac{3}{2}},
\end{align*}
which agrees the leading order large $r$ asymptotics of the Tracy-Widom GUE distribution \cite[Exercise 3.8.3]{AGZ10}. Note that the choice of normalizing constant $C$ in this argument is not arbitrary. Taking $C = \frac{1}{2}\partial_z^2 g(s,t)|_{z = \zeta}$ is consistent with the normalizing constants needed to see Tracy-Widom GUE limits in, for example, \cite[Theorem 1.6]{Johansson} (this is the case $\alpha,\beta \sim \delta_{\frac{1}{2}}$) and in \cite[Theorem 1.3]{BorodinCorwinFerrari}. In the latter case, this was shown to be the constant arising from the KPZ scaling theory in \cite{Spohn}. We also remark that the centering in this argument is likely not correct. As in \cite[Theorem 3]{GTW}, we expect that the correct centering should be  $n$ times the shape function with $\alpha$ and $\beta$ given by the empirical distribution of the parameters $\{a_i\}_{i=1}^{\lf ns \rf}$ and $\{b_j\}_{j=1}^{\lf nt \rf}$ rather than $n g(s,t)$. This new shape function is not random with respect to $\qmeas$ and converges to $g(s,t)$ for almost every realization of the environment. Continuity of the rate function then explains why this difference does not appear at the level of right tail large deviations. The same heuristic suggests that when $E[(a-\ubar{\alpha})^{-3}] < \infty$ or $E[(b-\ubar{\beta})^{-3}]<\infty$, we should expect KPZ type fluctuations in the critical directions $s/t = c_1$ or $s/t = c_2$, though we do not conjecture the precise limiting distribution in these cases. We also do not address the cases when $E[(a-\ubar{\alpha})^{-2}] < \infty$ but $E[(a-\ubar{\alpha})^{-3}] =\infty$ or  $E[(b-\ubar{\beta})^{-2}]<\infty$ but $E[(b-\ubar{\beta})^{-3}]=\infty$, though these are interesting questions.

\begin{thm}
\label{T2}
Suppose that $\alpha$ and $\beta$ are not both degenerate. For any $s, t > 0$, as $\epsilon \downarrow 0$, 
\[
\bbJ_{s, t}(g(s, t) + \epsilon) = 
\begin{cases}
\left(-s\E\left[\dfrac{1}{a-\ubar{\alpha}}\right]^2 + t\Var\left[\dfrac{1}{b+\ubar{\alpha}}\right]+t\E\left[\dfrac{1}{(b+\ubar{\alpha})^2}\right]\right)^{-1} \epsilon^2/2+o(\epsilon^2) &\text{ if } s/t < c_1 \\
\left(s\Var\left[\dfrac{1}{a+\zeta}\right] + t\Var\left[\dfrac{1}{b-\zeta}\right]\right)^{-1} \epsilon^{2}/2 + o(\epsilon^{2}) &\text{ if } c_1 \le s/t \le c_2 \\ 
\left(s\Var\left[\dfrac{1}{a+\ubar{\beta}}\right] +s\E\left[\dfrac{1}{(a+\ubar{\beta})^2}\right]- t\E\left[\dfrac{1}{b-\ubar{\beta}}\right]^2\right)^{-1} \epsilon^2/2+o(\epsilon^2) &\text{ if } s/t > c_2 
\end{cases}
\]
\end{thm}

We do not have any explicitly computable examples for which the regions $s/t \le c_1$ and $s/t \ge c_2$ are non-trivial, but we illustrate the results of the last two theorems with a numerical example. 
\begin{ex}
Choose $\alpha = 4(a-1)^31_{[1,2]}(a) da$ and $\beta = \delta_1$. We note that $\ubar{\alpha} = \ubar{\beta} = 1$. Explicit computation shows that
$\E\left[(a-1)^{-2}\right] = 2$,  $\E\left[(b-1)^{-2}\right] = \infty$, and $\E\left[(b+1)^{-2}\right] = \frac{1}{4}$. The linear region is then $\frac{s}{t} < \frac{1}{8}$. This is illustrated in Figure \ref{numericalshape} below.

\begin{figure}[H]
\begin{center}
\includegraphics[scale=.5]{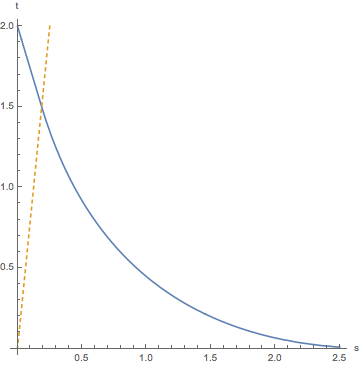}
\end{center}
\caption{\small{The level set $\{(s,t) : g(s,t) = 1\}$ (solid) and the boundary line $\frac{s}{t} = \frac{1}{8}$ (dashed).}}
\label{numericalshape}
\end{figure}

In Figure \ref{numericalasymp}, we plot numerical approximations of the rate functions against the small $\epsilon$ asymptotics in Theorems \ref{T3} and \ref{T2}. For example, frame (e) plots $\bfJ_{1,1}(g(1,1) + \epsilon)$ against $\frac{4}{3}(E (a+\zeta)^{-3} + \E(b-\zeta)^{-3})^{-\frac{1}{2}}\epsilon^{\frac{3}{2}}$, where $\zeta$ is the minimizer in (\ref{E28}).

\begin{figure}[H]
\centering
\begin{tabular}{@{}c c@{}}
\subfloat[Quenched linear, $t=10$]{
    \includegraphics[width=.5\textwidth]{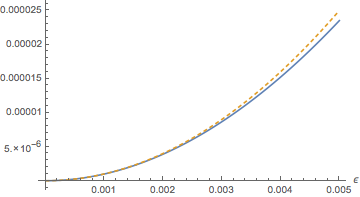}} &
\subfloat[Annealed linear, $t=10$]{
    \includegraphics[width=.5\textwidth]{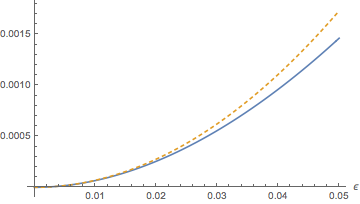}}
\end{tabular}
\end{figure}
\begin{figure}[H]
\centering
\begin{tabular}{@{}c c@{}}
\subfloat[Quenched boundary, $t=8$]{
    \includegraphics[width=.5\textwidth]{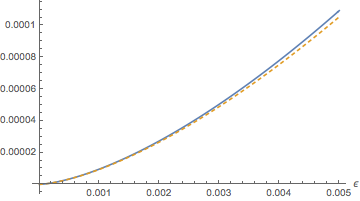} } &
\subfloat[Annealed boundary, $t =8$]{
    \includegraphics[width=.5\textwidth]{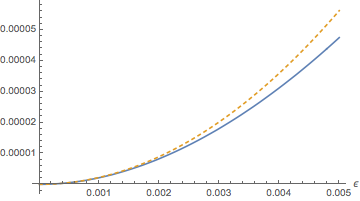}}
\end{tabular}
\end{figure}
\begin{figure}[H]
\begin{tabular}{@{}c c@{}}
\subfloat[Quenched concave $t=1$]{
    \includegraphics[width=.5\textwidth]{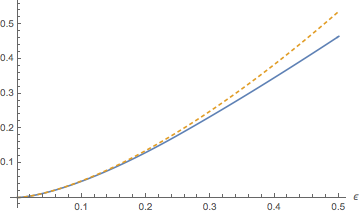}}&
\subfloat[Annealed concave $t=1$]{
    \includegraphics[width=.5\textwidth]{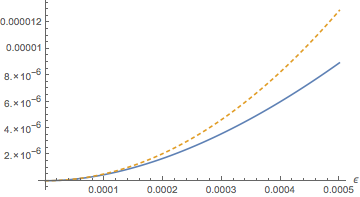}}   \\
  \end{tabular}
  \caption{\small{Plot of $\bbJ_{s,t}(g(s,t) + \epsilon)$ and $\bfJ_{s,t}(g(s,t) + \epsilon)$ (solid) and their $\epsilon \downarrow 0$ asymptotics (dashed) with $s = 1$.}}
\label{numericalasymp}
\end{figure}
\end{ex}

\section{Variational formulas for the Lyapunov exponents}
\label{S2}
Note from (\ref{E30}) that the probabilities under $\bfP_{\bfa, \bfb}^z$ and $\bbP^z$ of events generated by $\{W(i, 0): i \in \bbN\}$ make sense for any $z  > -\ubar{\alpha}$. Therefore, we permit ourselves to use notation $\bfP_{\bfa, \bfb}^z$ and $\bbP^z$ (and the corresponding expectations) for $z \ge \ubar{\beta}$ and, similarly, for $z \le -\ubar{\alpha}$ when we work only with $\{W(i, 0): i \in \bbN\}$ and $\{W(0, j): j \in \bbN\}$, respectively. 
\begin{lem}
\label{L4}
Let $\lambda \in \bbR$. Suppose that $z > -\ubar{\alpha}$ in (\ref{EQ1.3}), (\ref{EQ1}), and $z < \ubar{\beta}$ in (\ref{EQ1.4}) and (\ref{EQ1.1}) below. 
\begin{enumerate}[(a)]
\item 
$\mu$-a.s., for any $t > 0$,
\begin{align} 
\lim \limits_{n \rightarrow \infty} \frac{1}{n} \log \bfE^z_{\bfa, \bfb} \left[\exp\left(\lambda \sum_{i = 1}^{\lf nt \rf} W(i, 0)\right)\right] &= \begin{cases}t \E\left[\log\left(\dfrac{a+z}{a+z-\lambda}\right)\right] &\text{ if } \lambda \le \ubar{\alpha}+z \\ \infty &\text{ otherwise. }\end{cases} \label{EQ1.3} \\
\lim \limits_{n \rightarrow \infty} \frac{1}{n} \log \bfE^z_{\bfa, \bfb} \left[\exp\left(\lambda \sum_{i = 1}^{\lf nt \rf} W(0, i)\right)\right] &= \begin{cases}t \E\left[\log\left(\dfrac{b-z}{b-z-\lambda}\right)\right] &\text{ if } \lambda \le \ubar{\beta}-z \\ \infty &\text{ otherwise. }\end{cases}\label{EQ1.4}
\end{align}
\item For any $t > 0$, 
\begin{align}
\lim \limits_{n \rightarrow \infty} \frac{1}{n} \log \bbE^z \left[\exp\left(\lambda \sum_{i = 1}^{\lf nt \rf} W(i, 0)\right)\right]  &= \begin{cases}t \log\E\left[\dfrac{a+z}{a+z-\lambda}\right] &\text{ if } \lambda \le \ubar{\alpha}+z \\ \infty &\text{ otherwise. }\end{cases} \label{EQ1} \\
\lim \limits_{n \rightarrow \infty} \frac{1}{n} \log \bbE^z \left[\exp\left(\lambda \sum_{i = 1}^{\lf nt \rf} W(0, i)\right)\right] &= \begin{cases}t \log \E\left[\dfrac{b-z}{b-z-\lambda}\right] &\text{ if } \lambda \le \ubar{\beta}-z \\ \infty &\text{ otherwise. }\end{cases}\label{EQ1.1}
\end{align}
\end{enumerate}
\end{lem}
\begin{proof}
Using (\ref{E30}), we compute 
\begin{equation}
\label{E2}\bfE_{\bfa, \bfb}^z \left[e^{\lambda \sum \limits_{i=1}^{\lf nt \rf} W(i, 0)} \right] = 
\begin{cases}\displaystyle \prod \limits_{i=1}^{\lf nt \rf} \frac{a_i+z}{a_i+z-\lambda}
&\text{ if } \lambda < \min \limits_{1 \le i \le \lf nt \rf} a_i + z \\
\infty &\text{ otherwise. }
\end{cases}
\end{equation}
If $\lambda < \ubar{\alpha} + z$ then the first equality in (\ref{E2}) holds for all $n \in \bbN$ $\mu$-a.s and we have
\begin{equation}\label{E36}\E\left|\log \left(\frac{a+z}{a+z-\lambda}\right)\right| < \infty.\end{equation}
Hence, by the ergodicity of $\bfa$, 
\begin{equation}
\label{E4}
\lim_{n \rightarrow \infty} \frac{1}{n} \log \bfE^z_{\bfa, \bfb}\left[e^{\lambda \sum \limits_{i = 1}^{\lf nt \rf} W(i, 0)}\right] = \lim_{n \rightarrow \infty} \frac{1}{n} \sum \limits_{i=1}^{\lf nt \rf} \log \left(\frac{a_i+z}{a_i+z-\lambda}\right) = t \E \log \left(\frac{a+z}{a+z-\lambda}\right) \quad \mu\text{-a.s.}
\end{equation}
Moreover,  it follows from (\ref{E2}) that  
\begin{equation}
\label{E3}
\lim_{n \rightarrow \infty}\frac{1}{n} \log \bbE^z\left[e^{\lambda \sum \limits_{i = 1}^{\lf nt \rf} W(i, 0)}\right] = \lim_{n \rightarrow \infty}\frac{\lf nt \rf}{n} \log \E\left[\frac{a+z}{a+z-\lambda}\right] \rightarrow t \log \E\left[\frac{a+z}{a+z-\lambda}\right]. 
\end{equation}
Next, consider the case $\lambda = \ubar{\alpha}+z$. If (\ref{E36}) is in force, then both (\ref{E4}) and (\ref{E3}) still hold. Suppose now that (\ref{E36}) fails. By monotonicity, 
\begin{align*}
\liminf \limits_{n \rightarrow \infty} \frac{1}{n} \log \bfE^z_{\bfa, \bfb}\left[e^{\lambda \sum \limits_{i = 1}^{\lf nt \rf} W(i, 0)}\right] &\ge t \E \log \left(\frac{a+z}{a+z-\lambda'}\right) \quad \mu\text{-a.s.}\\
\liminf \limits_{n \rightarrow \infty} \frac{1}{n} \log \bbE^z\left[e^{\lambda \sum \limits_{i = 1}^{\lf nt \rf} W(i, 0)}\right] &\ge t \log \E\left[\frac{a+z}{a+z-\lambda'}\right]
\end{align*} 
for any $\lambda' < \lambda$. Letting $\lambda' \uparrow \lambda$ and monotone convergence yield 
\begin{equation}
\label{E37}
\lim \limits_{n \rightarrow \infty} \frac{1}{n} \log \bfE^z_{\bfa, \bfb}\left[e^{\lambda \sum \limits_{i = 1}^{\lf nt \rf} W(i, 0)}\right] = \lim \limits_{n \rightarrow \infty} \frac{1}{n} \log \bbE^z\left[e^{\lambda \sum \limits_{i = 1}^{\lf nt \rf} W(i, 0)}\right] = \infty.\end{equation}
Finally, consider the case $\lambda > \ubar{\alpha} + z$. Then, by the ergodicity of $\bfa$, there exists $i \in \bbN$ such that $\lambda \ge a_i+z$ and the second equality in (\ref{E2}) holds for large enough $n \in \bbN$ $\mu$-a.s. Hence, (\ref{E37}). 

We have verified (\ref{EQ1.3}) and (\ref{EQ1}). The proofs of (\ref{EQ1.4}) and (\ref{EQ1.1}) are similar. 
\end{proof}

Recall the basic properties of the Lyapunov exponents stated in Proposition \ref{P1}. For $s, t > 0$ and $\lambda \in \bbR$, define 
\begin{align*}
\bfL_{s, 0}(\lambda) = \lim_{t \downarrow 0} \bfL_{s, t}(\lambda) \qquad \bfL_{0, t}(\lambda) = \lim_{s \downarrow 0} \bfL_{s, t}(\lambda),
\end{align*}
where the limits exist by monotonicity. 
Define $\bbL_{s, 0}(\lambda)$ and $\bbL_{0, t}(\lambda)$ similarly. Also, for $k, l \in \bbZ_+$, let $\theta_{k, l}$ denote the shift given by $\omega(i, j) \mapsto \omega(i+k, j+l)$ for $i, j \in \bbN$ and $\omega \in \bbR^{\bbN^2}$. We next obtain a variational formula involving the Lyapunov exponents. 

\begin{lem}
\label{L5}
Let $z \in (-\ubar{\alpha}, \ubar{\beta})$ and $\lambda \in (0, \ubar{\beta}-z]$. Then 
\begin{align}
&\E\log \left(\frac{a+z+\lambda}{a+z}\right) + \E\log \left(\frac{b-z}{b-z-\lambda}\right)\label{E39}\\
&= 
\sup \limits_{0 \le t \le 1} \bigg\{\max \bigg\{\bfL_{t, 1}(\lambda) +(1-t)\E\log \left(\frac{a+z+\lambda}{a+z}\right), \bfL_{1, t}(\lambda)  + (1-t)\E\log \left(\frac{b-z}{b-z-\lambda}\right)\bigg\}\bigg\}. \nonumber
\end{align}
Also,
\begin{align}
&\log \E\left[\frac{a+z+\lambda}{a+z}\right] + \log\E\left[\frac{b-z}{b-z-\lambda}\right] \label{E59}\\
&= \sup \limits_{0 \le t \le 1} \bigg\{\max \bigg\{\bbL_{t, 1}(\lambda) +(1-t)\log\E\left[\frac{a+z+\lambda}{a+z}\right], \bbL_{1, t}(\lambda)  + (1-t)\log\E\left[\frac{b-z}{b-z-\lambda}\right]\bigg\}\bigg\}. \nonumber 
\end{align}
\end{lem}
\begin{proof}[Proof of (\ref{E39})]
We may assume that the left-hand side of (\ref{E39}) is finite. (This assumption fails only when $\lambda = \ubar{\beta}-z$ and $\E\log(b-\ubar{\beta}) = -\infty$ in which case (\ref{E39}) clearly holds).    

It follows from (\ref{E27}) and (\ref{E52}) that  
\begin{align*} 
\widehat{G}(n, n) = \max \limits_{1 \le k \le n} \{ \max \{G(n-k+1, n)\circ \theta_{k-1, 0} + \widehat{G}(k, 0), G(n, n-k+1)\circ \theta_{0, k-1}+ \widehat{G}(0, k)\}\}, 
\end{align*}
which leads to 
\begin{equation}
\label{E9}
\begin{aligned} 
\sum \limits_{1 \le j \le n} J(n, j) = \max \limits_{1 \le k \le n} \{ \max \{&G(n-k+1, n)\circ \theta_{k-1, 0}- \sum \limits_{k < i \le n} W(i, 0),  \\
 &G(n, n-k+1)\circ \theta_{0, k-1} - \sum \limits_{1 \le i \le n} W(i, 0) + \sum \limits_{1 \le j \le k} W(0, j)\}\}. 
\end{aligned}
\end{equation}
Also, note the identity 
\begin{equation}
\label{E43}
\frac{1}{\bfE_{\bfa, \bfb}^z\left[e^{-\lambda W(i, 0)}\right]} = \frac{a_i+z+\lambda}{a_i+z} = \bfE_{\bfa, \bfb}^{z+\lambda}\left[e^{\lambda W(i, 0)}\right] \text{ for } \lambda > 0 \text{ and } z > -\ubar{\alpha}. 
\end{equation}
Using the independence of weights under $\bfP_{\bfa, \bfb}^z$, Proposition \ref{P2}, (\ref{E9}) and (\ref{E43}), we obtain 
\begin{equation}
\label{E6}
\begin{aligned} 
&\bfE_{\bfa, \bfb}^{z+\lambda} \left[e^{\lambda \sum\limits_{1 \le i \le n}  W(i, 0)}\right] \cdot \bfE_{\bfa, \bfb}^{z} \left[e^{\lambda \sum\limits_{1 \le j \le n}  W(0, j)}\right]\\
&\ge  \max \bigg\{\bfE_{\tau_{k-1}(\bfa), \bfb} \left[e^{\lambda G(n-k+1, n)}\right] \cdot \bfE_{\bfa, \bfb}^{z+\lambda} \left[e^{\lambda \sum \limits_{1 \le i \le k} W(i, 0)}\right], \\
&\qquad\qquad \bfE_{\bfa, \tau_{k-1}(\bfb)} \left[e^{\lambda G(n, n-k+1)}\right] \cdot \bfE_{\bfa, \bfb}^z \left[e^{\lambda \sum \limits_{1 \le j \le k} W(0, j)}\right]\bigg\}. 
\end{aligned}
\end{equation}
Set $k = \lc n(1-t) \rc + 1$ for some $t \in (0, 1)$, apply logarithms to both sides and divide through by $n$ in (\ref{E6}).
It follows from Proposition \ref{P1} that 
\begin{align*}
\frac{1}{n}\log \bfE_{\tau_{k-1}(\bfa), \bfb} \left[e^{\lambda G(n-k+1, n)}\right] \rightarrow \bfL_{t, 1}(\lambda), \qquad
\frac{1}{n}\log \bfE_{\bfa, \tau_{k-1}(\bfb)} \left[e^{\lambda G(n, n-k+1)}\right] \rightarrow \bfL_{1, t}(\lambda)
\end{align*}
as $n \rightarrow \infty$ along suitable subsequences because $(\bfa, \bfb)$ is stationary and $\bfL$ is deterministic.
Hence, also using Lemma \ref{L4}, we obtain
\begin{equation}
\label{E13}
\begin{aligned}
&\E \log \left(\frac{a+z+\lambda}{a+z}\right) + \E\log \left(\frac{b-z}{b-z-\lambda}\right) \\
&\ge \max \bigg\{\bfL_{t, 1}(\lambda) +(1-t)\E\log \left(\frac{a+z+\lambda}{a+z}\right), \bfL_{1, t}(\lambda)  + (1-t)\E\log \left(\frac{b-z}{b-z-\lambda}\right)\bigg\}. 
\end{aligned}
\end{equation}
In particular, $\bfL$ is finite. By continuity, (\ref{E13}) holds with $t = 0$ and $t = 1$ as well. 

For the opposite inequality, introduce $L \in \bbN$ and let $n > L$ such that $\lc (l+1)n/L \rc > \lc ln/L\rc$ for $0 \le l < L$. Then, by (\ref{E9}) and nonnegativity of the weights,  
\begin{equation*}
\begin{aligned} 
\sum \limits_{1 \le j \le n} J(n, j) \le \max \limits_{1 \le l < L} \{ \max \{&G(\lf (L-l)n/L \rf, n) \circ \theta_{\lc ln/L \rc, 0} - \sum \limits_{\lc (l+1)n/L \rc < i \le n} W(i, 0),  \\
 &G(n, \lf (L-l)n/L \rf) \circ \theta_{0, \lc ln/L \rc} - \sum \limits_{1 \le i \le n} W(i, 0) + \sum \limits_{1 \le j \le \lc (l+1)n/L \rc} W(0, j)\}, 
\end{aligned}
\end{equation*}
which implies that  
\begin{equation}
\label{E10}
\begin{aligned} 
&\bfE_{\bfa, \bfb}^{z+\lambda} \left[e^{\lambda \sum\limits_{1 \le i \le n}  W(i, 0)}\right] \cdot \bfE_{\bfa, \bfb}^{z} \left[e^{\lambda \sum\limits_{1 \le j \le n}  W(0, j)}\right] \\ 
&\le \sum \limits_{0 \le l < L}  \bfE_{\tau_{\lc ln/L\rc}(\bfa), \bfb} \left[e^{\lambda G(\lf (L-l)n/L \rf, n)}\right] \cdot \bfE_{\bfa, \bfb}^{z+\lambda}\left[e^{\lambda \sum \limits_{i=1}^{\lc (l+1)n/L \rc} W(i, 0)}\right]  \\
&\qquad\ \ +\bfE_{\bfa, \tau_{\lc ln/L \rc}(\bfb)}\left[e^{\lambda G(n, \lf (L-l)n/L \rf)}\right] \cdot \bfE_{\bfa, \bfb}^z\left[e^{\lambda \sum \limits_{j = 1}^{\lc (l+1)n/L \rc} W(0, j)}\right]. 
\end{aligned}
\end{equation}
Taking logarithms leads to  
\begin{align*} 
&\log \bfE_{\bfa, \bfb}^{z+\lambda}\left[e^{\lambda \sum \limits_{1 \le i \le n} W(i, 0)}\right]  + \log \bfE_{\bfa, \bfb}^{z} \left[e^{\lambda \sum \limits_{1 \le j \le n}  W(0, j)}\right]\\
&\le \max \limits_{0 \le l < L}  \max \bigg\{\log \bfE_{\tau_{\lc ln/L\rc}(\bfa), \bfb} \left[e^{\lambda G(\lf (L-l)n/L \rf, n)}\right]+\log \bfE_{\bfa, \bfb}^{z+\lambda}\left[e^{\lambda \sum \limits_{i=1}^{\lc (l+1)n/L \rc} W(i, 0)}\right],   \\
 &\qquad\qquad\qquad \ \ \log \bfE_{\bfa, \tau_{\lc ln/L \rc}(\bfb)}\left[e^{\lambda G(n, \lf (L-l)n/L \rf)}\right] + \log \bfE_{\bfa, \bfb}^z\left[e^{\lambda \sum \limits_{j = 1}^{\lc (l+1)n/L \rc} W(0, j)}\right]\bigg\} + \log(2L). 
\end{align*}
Dividing through by $n$ and letting $n \rightarrow \infty$ along a suitable subsequential limit yield  
\begin{align*}
&\E\log \left(\frac{a+z+\lambda}{a+z}\right) + \E\log \left(\frac{b-z}{b-z-\lambda}\right)\\
&\le \max \limits_{0 \le l < L} 
\max \bigg\{\bfL_{1-l/L, 1}(\lambda) +\frac{l+1}{L} \E\log \left(\frac{a+z+\lambda}{a+z}\right), \bfL_{1, 1-l/L}(\lambda) +\frac{l+1}{L} \E\log \left(\frac{b-z}{b-z-\lambda}\right)\bigg\} \\
&\le \sup \limits_{0 \le t \le 1} \max \bigg\{\bfL_{t, 1}(\lambda) +(1-t)\E\log \left(\frac{a+z+\lambda}{a+z}\right), \bfL_{1, t}(\lambda) + (1-t) \E\log \left(\frac{b-z}{b-z-\lambda}\right)\bigg\}  \\
&+ \frac{1}{L}\left(\E\log \left(\frac{a+z+\lambda}{a+z}\right) + \E\log \left(\frac{b-z}{b-z-\lambda}\right)\right)
\end{align*}
Letting $L \rightarrow \infty$ completes the proof. 
\end{proof}
\begin{proof}[Proof of (\ref{E59})]
Some details will be skipped. We may assume that the left-hand side of (\ref{E59}) is finite. 

Using independence, we can rewrite (\ref{E6}) as 
\begin{equation}
\label{E41}
\begin{aligned}
\bfE_{\bfa, \bfb}^{z+\lambda} \left[e^{\lambda \sum\limits_{\substack{k < i \le n}}  W(i, 0)}\right] \cdot \bfE_{\bfa, \bfb}^{z} \left[e^{\lambda \sum\limits_{1 \le j \le n}  W(0, j)}\right]  &\ge \bfE_{\tau_{k-1}(\bfa), \bfb} \left[e^{\lambda G(n-k+1, n)}\right]\\
\bfE_{\bfa, \bfb}^{z+\lambda} \left[e^{\lambda \sum\limits_{\substack{1 \le i \le n}}  W(i, 0)}\right] \cdot \bfE_{\bfa, \bfb}^{z} \left[e^{\lambda \sum\limits_{k < j \le n}  W(0, j)}\right]  &\ge \bfE_{\bfa, \tau_{k-1}(\bfb)} \left[e^{\lambda G(n, n-k+1)}\right]\\
\end{aligned}
\end{equation}
The factors on the right-hand side are independent. Applying $\E$ yields
\begin{align} 
&\bbE^{z+\lambda} \left[e^{\lambda \sum\limits_{1 \le i \le n}  W(i, 0)}\right] \cdot \bbE^{z} \left[e^{\lambda \sum\limits_{1 \le j \le n}  W(0, j)}\right] \label{E8}\\
&\ge  \max \bigg\{\bbE \left[e^{\lambda G(n-k+1, n)}\right] \cdot \bbE^{z+\lambda} \left[e^{\lambda \sum \limits_{1 \le i \le k} W(i, 0)}\right], \bbE \left[e^{\lambda G(n, n-k+1)}\right] \cdot \bbE^z \left[e^{\lambda \sum \limits_{1 \le j \le k} W(0, j)}\right]\bigg\}, \nonumber 
\end{align}
where we rearranged terms using that $\{W(i, 0): i \in \bbN\}$ and $\{W(0, j): j \in \bbN\}$ are both i.i.d. under $\bbP^{z+\lambda}$ and $\bbP^{z}$. Then, (\ref{E8}) leads to $\ge$ half of (\ref{E59}) via Proposition \ref{P1} and Lemma \ref{L4}.  

For the $\le$ half of (\ref{E59}), suppose that $\lambda < \ubar{\beta}-z$ for the moment. Note the inequalities 
\begin{align*}
\bfE_{\bfa, \bfb}^{z+\lambda}[e^{\lambda W(i, 0)}] = \frac{a_i+z+\lambda}{a_i+z} \le \frac{\ubar{\alpha}+z+\lambda}{\ubar{\alpha}+z}, \qquad
\bfE_{\bfa, \bfb}^z[e^{\lambda W(0, j)}] = \frac{b_j-z}{b_j-z-\lambda} \le \frac{\ubar{\beta}-z}{\ubar{\beta}-z-\lambda}.
\end{align*}
It follows from these and (\ref{E10}) that  
\begin{equation}
\label{E10.1}
\begin{aligned} 
&\bfE_{\bfa, \bfb}^{z+\lambda} \left[e^{\lambda \sum\limits_{1 \le i \le n}  W(i, 0)}\right] \cdot \bfE_{\bfa, \bfb}^{z} \left[e^{\lambda \sum\limits_{1 \le j \le n}  W(0, j)}\right] \\
&\le \sum \limits_{0 \le l < L}  \left(\frac{\ubar{\alpha}+z+\lambda}{\ubar{\alpha}+z}\right)^{n/L+1}\bfE_{\tau_{\lc ln/L\rc}(\bfa), \bfb} \left[e^{\lambda G(\lf (L-l)n/L \rf, n)}\right] \cdot \bfE_{\bfa, \bfb}^{z+\lambda}\left[e^{\lambda \sum \limits_{i=1}^{\lc ln/L \rc} W(i, 0)}\right]\\
&\qquad \ \ +\left(\frac{\ubar{\beta}-z}{\ubar{\beta}-z-\lambda}\right)^{n/L+1} \bfE_{\bfa, \tau_{\lc ln/L \rc}(\bfb)}\left[e^{\lambda G(n, \lf (L-l)n/L \rf)}\right] \cdot \bfE_{\bfa, \bfb}^z\left[e^{\lambda \sum \limits_{j = 1}^{\lc ln/L \rc} W(0, j)}\right]. 
\end{aligned}
\end{equation}
The point of (\ref{E10.1}) is that the terms on the right-hand side are products of independent factors, which is not the case in (\ref{E10}).  
Applying $\log \E$, we obtain 
\begin{align*} 
&\log \bbE^{z+\lambda}\left[e^{\lambda \sum \limits_{1 \le i \le n} W(i, 0)}\right]  + \log \bbE^{z} \left[e^{\lambda \sum \limits_{1 \le j \le n}  W(0, j)}\right]\\
&\le \max \limits_{0 \le l < L}  \max \bigg\{(n/L+1) \log\left(\frac{\ubar{\alpha}+z+\lambda}{\ubar{\alpha}+z}\right) + \log \bbE \left[e^{\lambda G(\lf (L-l)n/L \rf, n)}\right]-\log \bbE^{z+\lambda}\left[e^{\lambda \sum \limits_{i=1}^{\lc ln/L \rc} W(i, 0)}\right],   \\
 &\qquad\qquad\qquad \ \ (n/L+1) \log \left(\frac{\ubar{\beta}-z}{\ubar{\beta}-z-\lambda}\right)+\log \bbE\left[e^{\lambda G(n, \lf (L-l)n/L \rf)}\right] + \log \bbE^z\left[e^{\lambda \sum \limits_{j = 1}^{\lc ln/L \rc} W(0, j)}\right]\bigg\}\\ 
&\qquad\qquad\qquad \ \ + \log(2L). 
\end{align*}
Divide through by $n$ and let $n \rightarrow \infty$. If we then send $L \rightarrow \infty$, the result is 
\begin{align*}
\log \E\left[\frac{a+z+\lambda}{a+z}\right] + \log\E\left[\frac{b-z}{b-z-\lambda}\right]
\le \sup \limits_{0 \le t \le 1} \bigg\{\max \bigg\{&\bbL_{t, 1}(\lambda) +(1-t)\log\E\left[\frac{a+z+\lambda}{a+z}\right],\\ &\bbL_{t, 1}(\lambda)  + (1-t)\log\E\left[\frac{b-z}{b-z-\lambda}\right]\bigg\}\bigg\}. 
\end{align*}
for all $\lambda < \ubar{\beta}-z$. The case $\lambda = \ubar{\beta}-z$ also follows because the right-hand side is nondecreasing in $\lambda$ and the left-hand side, due to monotone convergence, is continuous in $\lambda$ on $(0, \ubar{\beta}-z]$.  
\end{proof}
 
\begin{lem}
\label{L2}
For $\lambda > 0$, 
\begin{align}
&\bfL_{1, 0}(\lambda) = \E\log\left(\dfrac{a+\ubar{\beta}}{a+\ubar{\beta}-\lambda}\right) \quad \bfL_{0, 1}(\lambda) = \E\log\left(\dfrac{b+\ubar{\alpha}}{b+\ubar{\alpha}-\lambda}\right) \quad &&\text{ if } \lambda \le \ubar{\alpha} + \ubar{\beta} \\ &\bfL_{1, 0}(\lambda) = \bfL_{0, 1}(\lambda) = \infty \quad &&\text{ otherwise. }\label{EQ17}\\
&\bbL_{1, 0}(\lambda) = \log \E\left[\dfrac{a+\ubar{\beta}}{a+\ubar{\beta}-\lambda}\right] \quad 
\bbL_{0, 1}(\lambda) = \log \E\left[\dfrac{b+\ubar{\alpha}}{b+\ubar{\alpha}-\lambda}\right] \quad
&&\text{ if }\lambda \le \ubar{\alpha} + \ubar{\beta} \\ 
&\bbL_{1, 0}(\lambda) = \bbL_{0, 1}(\lambda) = \infty \quad &&\text{ otherwise. } \label{EQ17.2}
\end{align}
\end{lem}
\begin{proof}
Let $\epsilon > 0$. On the event $b_1 \le \ubar{\beta}+\epsilon$, which has positive $\mu$-probability, we have for $n \ge 1/\epsilon$ 
\[
\begin{aligned}
\frac{1}{n}\log \bfE_{\bfa, \bfb}[e^{\lambda G(n, \lf n\epsilon \rf)}] &\ge \frac{1}{n} \log \bfE_{\bfa, \bfb}\left[e^{\lambda \sum \limits_{1 \le i \le n} W(i, 1)}\right] \\
&= \begin{cases}\dfrac{1}{n} \sum \limits_{i=1}^n \dfrac{a_i+b_1}{a_i+b_1-\lambda} \quad &\text{ if } \lambda < \min \limits_{1 \le i \le n} a_i+b_1 \\ \infty \quad &\text{ otherwise}\end{cases} \\
&\ge \begin{cases}\dfrac{1}{n} \sum \limits_{i=1}^n \dfrac{a_i+\ubar{\beta}+\epsilon}{a_i+\ubar{\beta}+\epsilon-\lambda} \quad &\text{ if } \lambda < \min \limits_{1 \le i \le n} a_i+\ubar{\beta} +\epsilon \\ \infty \quad &\text{ otherwise}\end{cases} \\
&= \frac{1}{n} \log \bfE_{\bfa, \bfb}^{\ubar{\beta}+\epsilon}\left[e^{\lambda \sum \limits_{1 \le i \le n} W(i, 0)}\right]. 
\end{aligned}
\]
Then, by Lemma \ref{L4}, 
\[\bfL_{1, \epsilon}(\lambda) \ge \begin{cases}\E\left[\log\left(\dfrac{a+\ubar{\beta}+\epsilon}{a+\ubar{\beta}+\epsilon-\lambda}\right)\right] \quad &\text{ if } \lambda \le \ubar{\alpha}+\ubar{\beta}+\epsilon \\ \infty \quad &\text{ otherwise. }\end{cases}.\]
By monotone convergence, letting $\epsilon \downarrow 0$ yields 
\[\bfL_{1, 0}(\lambda) \ge \begin{cases}\E\left[\log\left(\dfrac{a+\ubar{\beta}}{a+\ubar{\beta}-\lambda}\right)\right] \quad &\text{ if } \lambda \le \ubar{\alpha}+\ubar{\beta}\\ \infty \quad &\text{ otherwise. }\end{cases}.\]
To complete the proof of (\ref{EQ17}), we need 
\begin{equation}
\label{E45}
\bfL_{1, 0}(\lambda) \le \E\left[\log\left(\dfrac{a+\ubar{\beta}}{a+\ubar{\beta}-\lambda}\right)\right]
\end{equation}
for $\lambda \in (0, \ubar{\alpha}+\ubar{\beta}]$. When $\lambda = \ubar{\alpha} + \ubar{\beta}$, we may assume that the right-hand side is finite. Then, $a_i > \ubar{\alpha}$ for $i \in \bbN$ a.s. and the argument in the paragraph of inequality (\ref{E6}) goes through with $z = -\ubar{\alpha}$ as well. Hence,  
\begin{equation}\label{E47}
\E\left[\log \left(\frac{a+z+\lambda}{a+z}\right)\right] + \E\left[\log \left(\frac{b-z}{b-z-\lambda}\right)\right] \ge \bfL_{1, t}(\lambda)+(1-t)\E\left[\log \left(\frac{b-z}{b-z-\lambda}\right)\right]\end{equation}
for $t \in [0, 1]$, $z \in [-\ubar{\alpha}, \ubar{\beta})$ and $\lambda \in (0, \ubar{\beta}-z]$, which simplifies to  
\begin{equation}\label{E44}\E \left[\log \left(\frac{a+z+\lambda}{a+z}\right)\right] + t \E\left[\log \left(\frac{b-z}{b-z-\lambda}\right)\right] \ge \bfL_{1, 0}(\lambda).\end{equation}
Setting $t = 0$ and $z = \ubar{\beta}-\lambda$ in (\ref{E44}) gives (\ref{E45}). The remaining cases are treated similarly. 
\end{proof}

\begin{cor}
\label{C1}
For $s, t > 0$, 
\begin{align*}\bfL_{s, t}(\ubar{\alpha}+\ubar{\beta}) &=  s\E\log\left(\dfrac{a+\ubar{\beta}}{a-\ubar{\alpha}}\right)+ t \E\log\left(\dfrac{b+\ubar{\alpha}}{b-\ubar{\beta}}\right).\\
\bbL_{s, t}(\ubar{\alpha}+\ubar{\beta}) &=  s \log\E\left[\dfrac{a+\ubar{\beta}}{a-\ubar{\alpha}}\right] + t \log \E\left[\dfrac{b+\ubar{\alpha}}{b-\ubar{\beta}}\right].\end{align*}
\end{cor}
\begin{proof}
By concavity and homogeneity, 
\begin{align}
\bfL_{s, t}(\ubar{\alpha} + \ubar{\beta}) \ge s\bfL_{1, 0}(\ubar{\alpha} + \ubar{\beta}) + t\bfL_{0, 1}(\ubar{\alpha} + \ubar{\beta}) = s\E\left[\log\left(\dfrac{a+\ubar{\beta}}{a-\ubar{\alpha}}\right)\right] + t \E\left[\log\left(\dfrac{b+\ubar{\alpha}}{b-\ubar{\beta}}\right)\right].\label{E46}
\end{align}
When the right-hand side is finite, the opposite inequality comes from (\ref{E47}). $\bbL_{s, t}(\ubar{\alpha}+\ubar{\beta})$ is computed similarly. 
\end{proof}

We will use the next lemma to recover the Lyapunov exponents from the variational formulas in Lemma \ref{L5}. 
\begin{lem}
\label{LInvert}
Let $a_0 < b_0$, $A: (a_0, b_0) \rightarrow \bbR$ be continuous and decreasing, $B: (a_0, b_0) \rightarrow \bbR$ be continuous and increasing. Let $L: [0, \infty)^2 \rightarrow \bbR$ be nondecreasing (in each variable), homogeneous, concave and continuous. Assume that 
\begin{align}
\label{varformula}
A(x) + B(x) = \sup \limits_{0 \le t \le 1} \{\max \{L(t, 1) + (1-t)A(x), L(1, t) + (1-t)B(x)\}\}  
\end{align}
for $a_0 < x < b_0$, $\lim_{x \uparrow b_0} A(x) = \lim_{t \downarrow 0} L(1, t)$ and $\lim_{x \downarrow a_0}B(x) = \lim_{s \downarrow 0} L(s, 1)$. Then 
\begin{align*}
L(s, t) = \inf_{a_0 < x < b_0} \{s A(x) + t B(x)\} \quad \text{ for } s, t > 0. 
\end{align*}
\end{lem}
\begin{proof}
The argument is the same as in \cite[Section 5]{Emrah} to prove Theorem 2.1. Assumption (\ref{varformula}) corresponds to Proposition 4.4 there, and $A(x) = \E[(a+x)^{-1}]$ and $B(x) = \E[(b-x)^{-1}]$ for $x \in (-\ubar{\alpha}, \ubar{\beta})$.
\end{proof}

\begin{proof}[Proof of Theorem \ref{T1}]
It follows from Lemma \ref{L2} that $\bfL_{s, t}(\lambda) = \infty$ for $\lambda > \ubar{\alpha}+\ubar{\beta}$. Fix $\lambda \in (0, \ubar{\alpha}+\ubar{\beta})$ and define 
\begin{align*}
A(z) = \E\left[\log \left(\frac{a+z+\lambda}{a+z}\right)\right] \text{ for } z > -\ubar{\alpha}, \qquad B(z) = \E\left[\log \left(\frac{b-z}{b-z-\lambda}\right)\right] \text{ for } z < \ubar{\beta}-\lambda.
\end{align*}
Lemma \ref{L5} states that
\[A(z) + B(z) = \sup \limits_{0 \le t \le 1} \{\max \{\bfL_{t, 1}(\lambda) + (1-t)A(z), \bfL_{1, t}(\lambda) + (1-t)B(z)\}\} \quad \text{ for } z \in (-\ubar{\alpha}, \ubar{\beta}-\lambda).\]
Note that $A$ and $B$ are continuous, $A$ is decreasing and $B$ is increasing. Moreover, by Lemma \ref{L2}, $A(\ubar{\beta}-\lambda) = \bfL_{1, 0}(\lambda)$ and $B(-\ubar{\alpha}) = \bfL_{0, 1}(\lambda)$. Also, $\bfL_{s, t}(\lambda)$ is finite and, by Proposition \ref{P1}, is nondecreasing, homogeneous, concave and continuous. Thus, by Lemma \ref{LInvert}, $\bfL_{s, t}(\lambda) = \inf_{-\ubar{\alpha} < z < \ubar{\beta}-\lambda} \{s A(z) + t B(z)\}$.
The endpoints can be included in the infimum, by monotone convergence. The proof of (\ref{E147}) is similar.
\end{proof}
We close this section with a proof of Theorem \ref{thm:statLyapunov}, which is similar to the arguments above.
\begin{proof}[Proof of Theorem \ref{thm:statLyapunov}]
We begin with the coupling
\begin{align*}
\hat{G}(\lf ns \rf, \lf nt \rf) &= \max_{1\leq k \leq \lf ns \rf} \left\{G(\lf ns \rf - k +1, \lf nt \rf)\circ \theta_{k-1,0} + \hat{G}(k,0)\right\} \\
&\vee \max_{1\leq k \leq \lf nt \rf}\left\{G(\lf ns \rf, \lf nt \rf - k + 1)\circ \theta_{0,k-1}  + \hat{G}(0,k)\right\}.
\end{align*}
Arguing with $\limsup$ and $\liminf$ and coarse graining as above, this leads to the variational problem
\begin{align*}
\bfL_{s,t}^z(\lambda) &= \max_{0\leq r \leq s}\left\{\bfL_{s-r,t}(\lambda) + r \E\left[\log \frac{a + z}{a + z - \lambda}\right] \right\} \vee \max_{0\leq u \leq t}\left\{\bfL_{s,t-u}(\lambda) + u \E\left[\log \frac{b-z}{b-z-\lambda}\right] \right\}.
\end{align*}
Substituting in the variational expression for $\bfL_{s,t}(\lambda)$, this leads to
\begin{align*}
\bfL_{s,t}^z(\lambda) &= \max_{0\leq r \leq s}\left\{\min_{\theta \in [-\ubar{\alpha},\ubar{\beta}-\lambda]}\left\{(s-r) \E \left[ \log \frac{a+\theta+\lambda}{a+\theta} \right] + t \E \left[\log \frac{b-\theta}{b-\theta-\lambda} \right]\right\} + r \E\left[\log \frac{a + z}{a + z - \lambda}\right] \right\} \\ 
&\vee \max_{0\leq u \leq t}\left\{\min_{\theta \in [-\ubar{\alpha},\ubar{\beta} -\lambda]}\left\{s \E\left[\log \frac{a+\theta+\lambda}{a+\theta}\right] + (t-u) \E\left[\log \frac{b-\theta}{b-\theta-\lambda}\right]\right\} + u \E\left[\log \frac{b-z}{b-z-\lambda}\right] \right\}.
\end{align*}
Applying a minimax theorem (for example \cite{Si58}), we obtain
\begin{align*}
\bfL_{s,t}^z(\lambda) &= \min_{\theta \in [-\ubar{\alpha},\ubar{\beta}-\lambda]}\left\{s \E \left[ \log \frac{a+\theta+\lambda}{a+\theta} \right] + t \E \left[\log \frac{b-\theta}{b-\theta-\lambda} \right]+ \max_{0\leq r \leq s}  r \E\left[\log \frac{(a + z)}{(a + z - \lambda)}\frac{(a+\theta)}{(a+\theta+\lambda)} \right] \right\} \\ 
&\vee \min_{\theta \in [-\ubar{\alpha},\ubar{\beta}-\lambda]}\left\{s \E\left[\log \frac{a+\theta+\lambda}{a+\theta}\right] + t \E\left[\log\frac{b-\theta}{b-\theta-\lambda}\right] +   \max_{0\leq u \leq t} u \E\left[\log \frac{(b-z)}{(b-z-\lambda)} \frac{(b-\theta-\lambda)}{(b-\theta)}\right]  \right\}.
\end{align*}
Write $(a+z - \lambda)(a+\theta + \lambda) = (a+z)(a+\theta) + \lambda(z-\theta-\lambda)$ to see that the inner maximum of the first term occurs at $r = s$ if $z - \lambda \leq \theta$ and $r=0$ if $z - \lambda \geq \theta$. Similarly, $\theta \mapsto (1-\lambda(b-\theta)^{-1})$ is a decreasing function, so the inner maximum of the second term occurs at $u= t$ for $\theta \leq z$ and at $u = 0$ for $\theta \geq z$. Breaking the first minimum over $[-\ubar{\alpha},\ubar{\beta} - \lambda]$ into a minimum over $[-\ubar{\alpha},z - \lambda]$ and a minimum over $[z - \lambda,\ubar{\beta}]$ and the second into a minimum over $[-\ubar{\alpha},z]$ and a minimum over $[z,\ubar{\beta}-\lambda]$, we obtain
\begin{align*}
&\min_{\theta \in [-\ubar{\alpha},\ubar{\beta}-\lambda]}\left\{s \E \left[ \log \frac{a+\theta+\lambda}{a+\theta} \right] + t \E \left[\log \frac{b-\theta}{b-\theta-\lambda} \right]+ \max_{0\leq r \leq s}  r \E\left[\log \frac{(a + z)}{(a + z - \lambda)}\frac{(a+\theta)}{(a+\theta+\lambda)} \right] \right\} \\
=& \left\{ s \E\left[ \log \frac{a+z}{a+z-\lambda}\right] + t \E\left[ \log \frac{b-z+\lambda}{b-z}\right]\right\}\wedge \min_{\theta \in [-\ubar{\alpha},z -\lambda]}\left\{s\E \left[ \log \frac{a+\theta+\lambda}{a+\theta} \right] + t \E \left[\log \frac{b-\theta}{b-\theta-\lambda} \right]\right\}
\end{align*}
and similarly, for the remaining term we have
\begin{align*}
& \min_{\theta \in [-\ubar{\alpha},\ubar{\beta}-\lambda]}\left\{s \E\left[\log \frac{a+\theta+\lambda}{a+\theta}\right] + t \E\left[\log\frac{b-\theta}{b-\theta-\lambda}\right] +   \max_{0\leq u \leq t} u \E\left[\log \frac{(b-z)}{(b-z-\lambda)} \frac{(b-\theta-\lambda)}{(b-\theta)}\right]  \right\}\\
=& \left\{s \E \left[ \log \frac{a+z+\lambda}{a+z}\right] + t \E\left[ \log \frac{b-z}{b-z-\lambda}\right]\right\}\wedge \min_{\theta \in [z, \ubar{\beta} - \lambda]}\left\{s\E \left[ \log \frac{a+\theta+\lambda}{a+\theta} \right] + t \E \left[\log \frac{b-\theta}{b-\theta-\lambda} \right]\right\}
\end{align*}
The function $\theta \mapsto s\E \left[ \log \frac{a+\theta+\lambda}{a+\theta} \right] + t \E \left[\log \frac{b-\theta}{b-\theta-\lambda} \right]$ is strictly convex with a unique minimizer. Note that the first terms in each of these minima are the values of this function evaluated at $\theta = z-\lambda$ and $\theta = z$. The result follows from strict convexity by considering whether the minimizer lies in $[-\ubar{\alpha},z], [z,z-\lambda],$ or $[z-\lambda,\ubar{\beta} - \lambda]$.
\end{proof}

\section{Extremizers of the variational problems} \label{Sreg}
In this section, we derive some regularity properties of $\bfL, \bbL, \bfJ$ and $\bbJ$ by studying the extremizers of their variational representations. The next two lemmas describe the minimizers of (\ref{E146}) and (\ref{E147}). See Figure \ref{F2} for an illustration. 

\begin{lem}
\label{L3}
Fix $s, t > 0$ and define $F = F(z, \lambda)$ for $0 < \lambda < \ubar{\alpha}+\ubar{\beta}$ and $-\ubar{\alpha} \le z \le \ubar{\beta}-\lambda$ by 
\begin{align}
F(z, \lambda) = s\E\log\left(\frac{a+z+\lambda}{a+z}\right) + t\E\log\left(\frac{b-z}{b-z-\lambda}\right). \label{E116}
\end{align}
For each $\lambda \in (0, \ubar{\alpha}+\ubar{\beta})$, there exists a unique $\zs = \zs(\lambda) \in [-\ubar{\alpha}, \ubar{\beta}-\lambda]$ such that
$\bfL_{s, t}(\lambda) = F(\zs, \lambda)$. We have $\zs = -\ubar{\alpha}$ if and only if 
\begin{align}
\label{E141}
-s\E\left[\frac{1}{(a-\ubar{\alpha}+\lambda)(a-\ubar{\alpha})}\right] + t \E\left[\frac{1}{(b+\ubar{\alpha}-\lambda)(b+\ubar{\alpha})}\right] \ge 0, \end{align}
and $\zs = \ubar{\beta}-\lambda$ if and only if 
\begin{align}
\label{E142}
-s\E\left[\frac{1}{(a+\ubar{\beta})(a+\ubar{\beta}-\lambda)}\right] + t \E\left[\frac{1}{(b-\ubar{\beta})(b-\ubar{\beta}+\lambda)}\right] \le 0. 
\end{align}
Define $\lambda_1 = \inf \{\lambda \in (0, \ubar{\alpha}+\ubar{\beta}): \text{(\ref{E141}) holds.}\} \wedge (\ubar{\alpha}+\ubar{\beta})$ and $\lambda_2 = \inf \{\lambda \in (0, \ubar{\alpha}+\ubar{\beta}): \text{(\ref{E142}) holds.}\} \wedge (\ubar{\alpha}+\ubar{\beta})$. Then $\zs = -\ubar{\alpha}$ if and only if $\lambda \ge \lambda_1$, and $\zs = \ubar{\beta}-\lambda$ if and only if $\lambda \ge \lambda_2$. For $0 < \lambda < \lambda_0 = \lambda_1 \wedge \lambda_2$, we have $\partial_z F(\zs, \lambda) = 0$. Moreover, $\zs$ is continuous on $(0, \ubar{\alpha}+\ubar{\beta})$ and continuously differentiable on $(0, \ubar{\alpha}+\ubar{\beta}) \smallsetminus \{\lambda_0\}$. We have $-1 < \zs' < 0$ for $0 < \lambda < \lambda_0$, $\lim_{\lambda \downarrow 0} \zs = \zeta(s, t)$ and $\lim_{\lambda \uparrow \ubar{\alpha}+\ubar{\beta}} \zs = -\ubar{\alpha}$. 
\end{lem}

\begin{lem}
\label{L9}
Lemma \ref{L3} holds verbatim if $\bfL_{s, t}$, (\ref{E116}), (\ref{E141}) and (\ref{E142}) are replaced with $\bbL_{s, t}$, 
\begin{align}
F(z, \lambda) = s\log \E\left[\frac{a+z+\lambda}{a+z}\right] + t\log\E\left[\frac{b-z}{b-z-\lambda}\right]\label{E143} \\
-s \frac{\E\left[\dfrac{1}{(a-\ubar{\alpha})^2}\right]}{\E\left[\dfrac{a-\ubar{\alpha}+\lambda}{a-\ubar{\alpha}}\right]} + t \frac{\E\left[\dfrac{1}{(b+\ubar{\alpha}-\lambda)^2}\right]}{\E\left[\dfrac{b+\ubar{\alpha}}{b+\ubar{\alpha}-\lambda}\right]} \ge 0 \label{E144}\\
-s \frac{\E\left[\dfrac{1}{(a+\ubar{\beta}-\lambda)^2}\right]}{\E\left[\dfrac{a+\ubar{\beta}}{a+\ubar{\beta}-\lambda}\right]} + t \frac{\E\left[\dfrac{1}{(b-\ubar{\beta})^2}\right]}{\E\left[\dfrac{b-\ubar{\beta}+\lambda}{b-\ubar{\beta}}\right]} \le 0 \label{E145}, 
\end{align}
respectively. Here, the left-hand sides of (\ref{E144}) and (\ref{E145}) are interpreted as $-\infty$ and $\infty$ when $\E[(a-\ubar{\alpha})^{-1}] = \infty$ and $\E[(b-\ubar{\beta})^{-1}] = \infty$, respectively. 
\end{lem}

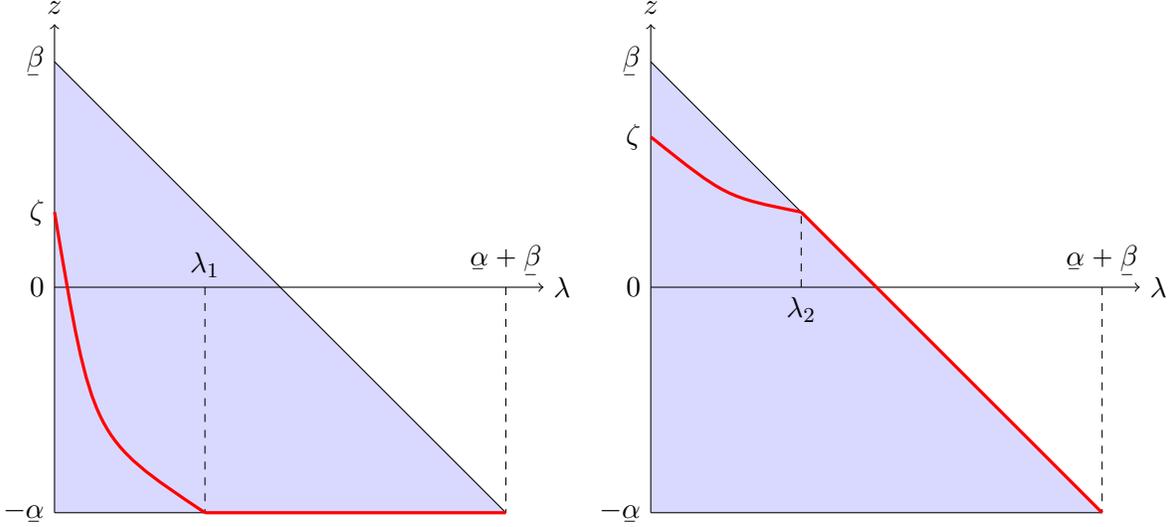
\begin{figure}[h!]
\centering
\begin{tikzpicture}[scale = 1]
\fill[blue!15](-1, 5)--(5, -1)--(-1, -1);
\draw[<-](-1, 5.5)node[above]{$z$}--(-1, -1);
\draw[->](-1, 2)--(5.5, 2)node[right]{$\lambda$};
\draw (-1, 2)node[left]{$0$};
\draw(-1, 5)node[left]{$\ubar{\beta}$};
\draw(-1, -1)node[left]{$-\ubar{\alpha}$};
\draw(5, 2)node[above]{$\ubar{\alpha}+\ubar{\beta}$};
\draw(-1, -1)--(5, -1);
\draw[dashed](5, 2)--(5, -1);
\draw(-1,5)--(5, -1);
\draw(1, 2)node[above]{$\lambda_1$};
\draw[dashed](1,2)--(1, -1);
\draw[red, very thick](1, -1)--(5, -1);
\draw(-1, 3)node[left]{$\zeta$};
\draw[red, very thick](-1, 3)..controls (-0.5, 0)..(1, -1);
\end{tikzpicture}
\begin{tikzpicture}[scale = 1]
\fill[blue!15](-1, 5)--(5, -1)--(-1, -1);
\draw[<-](-1, 5.5)node[above]{$z$}--(-1, -1);
\draw[->](-1, 2)--(5.5, 2)node[right]{$\lambda$};
\draw (-1, 2)node[left]{$0$};
\draw(-1, 5)node[left]{$\ubar{\beta}$};
\draw(-1, -1)node[left]{$-\ubar{\alpha}$};
\draw(5, 2)node[above]{$\ubar{\alpha}+\ubar{\beta}$};
\draw(-1, -1)--(5, -1);
\draw[dashed](5, 2)--(5, -1);
\draw(-1,5)--(5, -1);
\draw(1, 2)node[below]{$\lambda_2$};
\draw[dashed](1,2)--(1, 3);
\draw[red, very thick](1, 3)--(5, -1);
\draw(-1, 4)node[left]{$\zeta$};
\draw[red, very thick](-1, 4)..controls (0, 3.2)..(1, 3);
\end{tikzpicture}
\caption{\small{Sketches of the graph of the minimizers in (\ref{E146}) and (\ref{E147}) assuming (\ref{E141}) and (\ref{E144}), respectively (left) and assuming (\ref{E142}) and (\ref{E145}), respectively (right).}}
\label{F2}
\end{figure}

\begin{proof}[Proof of Lemma \ref{L3}]
Since $\partial_z^2 F > 0$, the existence and the uniqueness of $\zs$ follows. Also, $\zs = -\ubar{\alpha}$ if and only if $\partial_z F(-\ubar{\alpha}, \lambda) \ge 0$, which is (\ref{E141}). We note that $\partial_z F(-\ubar{\alpha}, \lambda) = -\infty$ if $\E[(a-\ubar{\alpha})^{-1}] = \infty$ and, otherwise, 
$\lambda^{-1} \partial_z F(-\ubar{\alpha}, \lambda)$ is a continuous, increasing function of $\lambda \in (0, \ubar{\alpha}+\ubar{\beta})$. Therefore, $\zs = -\ubar{\alpha}$ if and only if $\lambda \ge \lambda_1$. 
We similarly observe (\ref{E142}) and the equivalence of $\zs = \ubar{\beta}-\lambda$ and $\lambda \ge \lambda_2$. (Because $\partial_z F$ is increasing in $z$,  we cannot have $\lambda_1$ and $\lambda_2$ both less than $\ubar{\alpha}+\ubar{\beta}$).   

When $\lambda < \lambda_0$, the minimizer is the unique $\zs \in (-\ubar{\alpha}, \ubar{\beta}-\lambda)$ satisfying 
\begin{align}
\label{E112}
\partial_z F(\zs, \lambda) = 0. 
\end{align}
By the implicit function theorem, $\zs$ is continuously differentiable for $0 < \lambda < \lambda_0$ with derivative 
\begin{align}
\label{E124}
\zs'(\lambda) = - \frac{\partial_\lambda \partial_z F(\zs, \lambda)}{\partial_z^2 F(\zs, \lambda)}. 
\end{align}
Observing that  
\begin{align*}
&\partial_{\lambda} \partial_z F(\zs, \lambda) > -s\E\left[\frac{1}{(a+\zs)(a+\zs+\lambda)}\right]+t\E\left[\frac{1}{(b-\zs)(b-\zs-\lambda)}\right] = \lambda^{-1} \partial_z F(\zs, \lambda) = 0\\
&\partial_{z}^2 F(\zs, \lambda)-\partial_{\lambda}\partial_z F(\zs, \lambda) = s\E\left[\frac{1}{(a+\zs)^2}\right] - t\E\left[\frac{1}{(b-\zs)^2}\right] \\
&\qquad\qquad\qquad\qquad\qquad\ \ \ \ > s\E\left[\frac{1}{(a+\zs)(a+\zs+\lambda)}\right]-t\E\left[\frac{1}{(b-\zs)(b-\zs-\lambda)}\right] = 0, 
\end{align*}
we conclude that $-1 < \zs'(\lambda) < 0$. In particular, $\zs$ is monotone and has limits as $\lambda \downarrow 0$ and $\lambda \uparrow \lambda_0$. We also have continuous differentiability of $\zs$ for $\lambda > \lambda_0$. Now, supposing $\lambda_0 \in (0, \ubar{\alpha}+\ubar{\beta})$, we show that $\zs$ is continuous at $\lambda_0$. Letting $\lambda \uparrow \lambda_0$ in (\ref{E112}), we obtain 
\begin{align}
\label{E113}
\partial_z F(\lim_{\lambda \uparrow \lambda_0} \zs(\lambda), \lambda_0) = 0.
\end{align}
Since the minimizer occurs at the boundary when $\lambda = \lambda_0$, we deduce from (\ref{E113}) that 
$\lim_{\lambda \uparrow \lambda_1} \zs(\lambda) = -\ubar{\alpha}$ and $\lim_{\lambda \uparrow \lambda_2} \zs(\lambda) = \ubar{\beta}-\lambda_2$ when $\lambda_0 = \lambda_1$ and $\lambda_0 = \lambda_2$, respectively. 

Since $\zs(\lambda) \in [-\ubar{\alpha}, \ubar{\beta}-\lambda]$, we have $\lim_{\lambda \uparrow \ubar{\alpha}+\ubar{\beta}} \zs(\lambda) = -\ubar{\alpha}$. Set $\zs(0) = \lim_{\lambda \downarrow 0} \zs(\lambda)$. To calculate this limit, we consider several cases. If $\lambda_0 > 0$ then we can let $\lambda \downarrow 0$ in (\ref{E112}) and obtain 
\begin{align*}0 = \partial_z F(\zs(0), 0) = -s \E\left[\frac{1}{(a+\zs(0))^2}\right] + t\E\left[\frac{1}{(b-\zs(0))^2}\right] = \partial_z g_{\zs(0)}(s, t),\end{align*} 
which implies $\zs(0) = \zeta$. If $\lambda_1 = 0$ then $\partial_z F(-\ubar{\alpha}, 0) = \partial_z g_{-\ubar{\alpha}}(s, t) \ge 0$ and if $\lambda_2 = 0$ then $\partial_z F(\ubar{\beta}, 0) = \partial_z g_{\ubar{\beta}}(s, t) \le 0$. Hence, we get $\zeta = -\ubar{\alpha} = \zs(0)$ and $\zeta = \ubar{\beta} = \zs(0)$, respectively. 
\end{proof}
We omit the proof of Lemma \ref{L9} which is similar to that of Lemma \ref{L3}. 
\begin{lem}
\label{L8}
For each $s, t > 0$, $\bfL_{s, t}$ is continuously differentiable on $[0, \ubar{\alpha}+\ubar{\beta})$ and $\bfL_{s, t}'(0) = g(s, t)$. Furthermore, $\bfL_{s, t}'$ is continuously differentiable on $(0, \ubar{\alpha}+\ubar{\beta}) \smallsetminus \{\lambda_0\}$ and $\bfL_{s, t}'' > 0$. The same statements also hold for $\bbL_{s, t}$. 
\end{lem}
\begin{proof}
Let us write $L$ for $\bfL_{s, t}$ and $F = F(z, \lambda)$ be given by (\ref{E143}). Using Lemma \ref{L3}, we compute
\begin{align}
L'(\lambda) &= \partial_z F(\zs, \lambda) \zs'(\lambda) + \partial_\lambda F(\zs, \lambda) = s\E \left[\dfrac{1}{a+\zs+\lambda}\right] + t \E \left[\dfrac{1}{b-\zs-\lambda}\right]
\label{E115}
\end{align}
for $0 < \lambda < \lambda_0$. Differentiating again, we obtain 
\begin{align*}
L''(\lambda) &= \partial_z \partial_\lambda F(\zs, \lambda) \zs'(\lambda) + \partial_{\lambda}^2 F(\zs, \lambda)= \frac{\partial_z^2 F(\zs, \lambda) \partial_{\lambda}^2 F(\zs, \lambda)-\partial_z \partial_\lambda F(\zs, \lambda)^2}{\partial_z^2 F(\zs, \lambda)} > 0, 
\end{align*}
where the inequality comes from $\partial_z^2 F(\zs, \lambda) > \partial_{\lambda} \partial_z F(\zs, \lambda)$ and $\partial_{\lambda}^2 F = \partial_{\lambda} \partial_z F$. For $\lambda > \lambda_1$,  
\begin{align}
L'(\lambda) &= s\E \left[\dfrac{1}{a-\ubar{\alpha}+\lambda}\right] + t \E \left[\dfrac{1}{b+\ubar{\alpha}-\lambda}\right] \label{E119}\\
L''(\lambda) &= -s\E \left[\dfrac{1}{(a-\ubar{\alpha}+\lambda)^2}\right] + t \E \left[\dfrac{1}{(b+\ubar{\alpha}-\lambda)^2}\right] > \partial_{z}F(-\ubar{\alpha}, \lambda) > 0.  
\end{align}
Also, for $\lambda > \lambda_2$, 
\begin{align}
L'(\lambda) &= s\E \left[\dfrac{1}{a+\ubar{\beta}-\lambda}\right] + t \E \left[\dfrac{1}{b-\ubar{\beta}+\lambda}\right] \label{E120}\\
L''(\lambda) &= s\E \left[\dfrac{1}{(a+\ubar{\beta}-\lambda)^2}\right] - t \E \left[\dfrac{1}{(b-\ubar{\beta}+\lambda)^2}\right] > -\partial_{z} F(\ubar{\beta}-\lambda, \lambda) > 0. 
\end{align}
We have verified that $L$ is continuously differentiable on $(0, \ubar{\alpha}+\ubar{\beta}) \smallsetminus \{\lambda_0\}$ and $L'$ is increasing. 

We next note that $L$ is also continuously differentiable at $\lambda_0$ when $\lambda_0 \in (0, \ubar{\alpha}+\ubar{\beta})$, for which it suffices to check that the left and right limits of $L'$ at $\lambda_0$ match. First, we consider the case $\lambda_1 \in (0, \ubar{\alpha}+\ubar{\beta})$. Then, as $\lambda \uparrow \lambda_1$, (\ref{E115}) tends to $s\E[(a-\ubar{\alpha}+\lambda_1)^{-1}]+t\E[(b-\ubar{\alpha}-\lambda_1)^{-1}]$, which equals the $\lambda \downarrow \lambda_1$ limit of (\ref{E119}). Now, suppose that $\lambda_2 \in (0, \ubar{\alpha}+\ubar{\beta})$. Then, as $\lambda \uparrow \lambda_2$, (\ref{E115}) tends to $s\E[(a+\ubar{\beta})^{-1}]+t\E[(b-\ubar{\beta})^{-1}]$, which is the same as 
\begin{align*}
s\E \left[\dfrac{1}{a+\ubar{\beta}-\lambda_2}\right] + t \E \left[\dfrac{1}{b-\ubar{\beta}+\lambda_2}\right] + \partial_z F(\ubar{\beta}-\lambda_2, \lambda_2) 
&=  s\E \left[\dfrac{1}{a+\ubar{\beta}-\lambda_2}\right] + t \E \left[\dfrac{1}{b-\ubar{\beta}+\lambda_2}\right], 
\end{align*}
the $\lambda \downarrow \lambda_2$ limit of (\ref{E120}). 

We next calculate $L'(0) = \lim_{\lambda \downarrow 0} L'(\lambda)$. If $\lambda_0 > 0$ then $\lambda \downarrow 0$ limit of (\ref{E115}) gives 
\begin{align*}L'(0) = s\E \left[\dfrac{1}{a+\zeta}\right] + t \E \left[\dfrac{1}{b-\zeta}\right] = g(s, t).\end{align*}
In the cases $\lambda_1 = 0$ and $\lambda_2$ then $\zeta = -\ubar{\alpha}$ and $\zeta = \ubar{\beta}$, respectively. Hence, letting $\lambda \downarrow 0$ in (\ref{E119}) and (\ref{E120}), respectively, we still obtain $L'(0) = g(s, t)$. 

The asserted properties of $\bbL$ are proved similarly. 
\end{proof}
Since $\bfL_{s, t}'$ increasing, $\bfL_{s, t}'(\lambda)$ has a limit (possibly $\infty$) as $\lambda \uparrow \ubar{\alpha}+\ubar{\beta}$, which we denote by $\bfL_{s,t}'(\ubar{\alpha}+\ubar{\beta})$. Similarly, let us write $\bbL_{s, t}'(\ubar{\alpha}+\ubar{\beta})$ for $\lim_{\lambda \uparrow \ubar{\alpha}+\ubar{\beta}} \bbL_{s, t}'(\lambda)$. The precise values of these limits will be needed in the next section.
\begin{cor}
\label{L10}
Fix $s, t > 0$. 
\begin{align*}
\bfL_{s, t}'(\ubar{\alpha}+\ubar{\beta}) &= \begin{cases} s \E\left[\dfrac{1}{a-\ubar{\alpha}}\right] + t\E\left[\dfrac{1}{b+\ubar{\alpha}}\right] \quad \text{ if } -s \E\left[\dfrac{1}{(a-\ubar{\alpha})(a+\ubar{\beta})}\right] + t\E\left[\dfrac{1}{(b+\ubar{\alpha})(b-\ubar{\beta})}\right] \le 0 \\ 
s \E\left[\dfrac{1}{a+\ubar{\beta}}\right] + t\E\left[\dfrac{1}{b-\ubar{\beta}}\right] \quad \text{ otherwise. }\end{cases} \\
\bbL_{s, t}'(\ubar{\alpha}+\ubar{\beta}) &= \begin{cases}s\dfrac{\E\left[\dfrac{a+\ubar{\beta}}{(a-\ubar{\alpha})^2}\right]}{\E\left[\dfrac{a+\ubar{\beta}}{a-\ubar{\alpha}}\right]}+t \dfrac{\E\left[\dfrac{1}{b-\ubar{\beta}}\right]}{\E\left[\dfrac{b+\ubar{\alpha}}{b-\ubar{\beta}}\right]} \quad \text{ if }-s\dfrac{\E\left[\dfrac{1}{(a-\ubar{\alpha})^2}\right]}{\E\left[\dfrac{a+\ubar{\beta}}{a-\ubar{\alpha}}\right]}+t \dfrac{\E\left[\dfrac{1}{(b-\ubar{\beta})^2}\right]}{\E\left[\dfrac{b+\ubar{\alpha}}{b-\ubar{\beta}}\right]} \le 0\\ 
s\dfrac{\E\left[\dfrac{1}{a-\ubar{\alpha}}\right]}{\E\left[\dfrac{a+\ubar{\beta}}{a-\ubar{\alpha}}\right]}+t \dfrac{\E\left[\dfrac{b+\ubar{\alpha}}{(b-\ubar{\beta})^2}\right]}{\E\left[\dfrac{b+\ubar{\alpha}}{b-\ubar{\beta}}\right]}
\quad \text{ otherwise. }\end{cases}
\end{align*}
\end{cor}
The next lemma establishes continuous differentiability of $\bfJ_{s, t}(r)$ and $\bbJ_{s, t}(r)$ and shows that these functions are linear in $r$ for $r > \bfL_{s, t}'(\ubar{\alpha}+\ubar{\beta})$ and $r > \bbL_{s, t}'(\ubar{\alpha}+\ubar{\beta})$, respectively. 
\begin{lem}
\label{L1}
Fix $s, t > 0$. For each $r \ge g(s, t)$, there exists a unique $\lam(r) \in [0, \ubar{\alpha}+\ubar{\beta}]$ such that 
$\bfJ_{s, t}(t) = \lam r - \bfL_{s, t}(\lam)$. Moreover, $\bfJ_{s, t}$ is continuously differentiable and $\bfJ_{s, t}'(r) = \lam(r)$ for $r \ge g(s, t)$. If $r > g(s,t)$, then $\lam > 0$. If $r \ge \bfL_{s, t}'(\ubar{\alpha}+\ubar{\beta})$ then $\lam = \ubar{\alpha}+\ubar{\beta}$, while if $r \in [g(s, t), \bfL_{s, t}'(\ubar{\alpha}+\ubar{\beta}))$ then $\bfL_{s, t}'(\lam) = r$. The same statements hold if we replace $\bfJ_{s, t}$ and $\bfL_{s, t}$ with $\bbJ_{s, t}$ and $\bbL_{s, t}$, respectively. 
\end{lem}
\begin{proof}
We have $J(r) = \sup_{0 < \lambda < \ubar{\alpha}+\ubar{\beta}} \{\lambda r - L(\lambda)\}$, where $(L, J)$ pair refers to either $(\bfL_{s, t}, \bfJ_{s, t})$ or $(\bbL_{s, t}, \bbJ_{s, t})$. The $\lambda$-derivative of the function inside the supremum is $r-L'(\lambda)$. By Lemma \ref{L8}, $L'$ is continuous and increasing from $g(s, t)$ to the limit $L'(\ubar{\alpha}+\ubar{\beta})$ on $(0, \ubar{\alpha}+\ubar{\beta})$. It follows that the unique maximizer $\lam$ is at $\ubar{\alpha}+\ubar{\beta}$ if $r \ge L'(\ubar{\alpha}+\ubar{\beta})$ and at $(L')^{-1}(r)$, otherwise. In addition, $\lam$ is increasing and continuous on $[g(s, t), +\infty)$. Since $L'$ is differentiable and has nonzero derivative for $\lambda \in (0, \ubar{\alpha}+\ubar{\beta}) \smallsetminus {\lambda_0}$, whenever $r \neq L'(\lambda_0)$, we have $J'(r) = \lam(r) + \lam'(r)r-L'(\lam)\lam'(r) = \lam(r)$. Then continuity of $\lam$ implies that $J$ is continuously differentiable for all $r \ge g(s, t)$ including $L'(\lambda_0)$ when $\lambda_0 \in (0, \ubar{\alpha}+\ubar{\beta})$. 
\end{proof}

\begin{proof}[Proof of Theorem \ref{T4}]
This theorem is included in the preceding lemma. 
\end{proof}

\section{Left tail estimates}
\label{S4}

We now estimate the left tail in both the quenched and annealed settings. The first result shows that in the quenched case, the rate $n$ large deviation rate function will be trivial for deviations to the left of the shape function $g(s,t)$. This proof is based on the proof of \cite[Theorem 4.1]{Seppalainen3}, which was adapted from an argument in \cite{Kesten}.

\begin{proof}[Proof of Lemma \ref{quenchedlb}]
First, fix $s,t,\epsilon >0$ and rational. Take $m \in \mathbb{N}$ large enough that \\
$m^{-1} \mathbb{E} \hspace{.25pc} G(\lfloor m s \rfloor, \lfloor m t \rfloor) \geq g(s,t) - \frac{\epsilon}{2}$. We coarse grain the lattice into pairwise disjoint translates of the set $\{1, \dots , \lfloor m s \rfloor\} \times \{1, \dots, \lfloor m t \rfloor\}$. Toward this end, define
\begin{align*}
A_{a,b}^{k,\ell} &= \{1 + a, \dots, a + k\} \times \{1 + b, \dots, \ell + b\}, &\qquad B_i^j = A_{(j+ i)\lfloor m s \rfloor, j \lfloor m t \rfloor}^{\lfloor mx \rfloor, \lfloor my \rfloor}.
\end{align*}
Take $n$ large and let $L = \lfloor \frac{n}{m} - \lfloor \sqrt{n} \rfloor - 2\rfloor$. For each such $k \leq \lfloor \sqrt{n} \rfloor$, define a diagonal by $D_k = \bigcup_{j=0}^{L} B_k^j$. We observe that the passage time from the bottom left corner of $B_i^j$ to the top-right corner of $B_i^j$, $G_{i,j} \equiv G(\lfloor ms \rfloor, \lfloor mt \rfloor) \circ \tau_{(i+j)\lfloor ms \rfloor, j \lfloor mt \rfloor}$, has the same distribution as $G_{0,0}$ under $\ameas$. Moreover, if $(i_1,j_1) \neq (i_2, j_2)$, then $B_{i_1}^{j_1} \bigcap B_{i_2}^{j_2} = \emptyset$ and consequently $\{G_{i,j}\}_{i,j\geq 0}$ forms an independent family under $\qmeas$. 

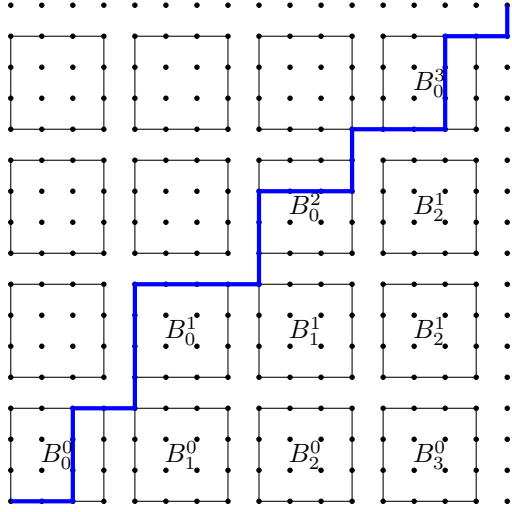
\begin{figure} [H]
\centering
\begin{tikzpicture}[scale=1.65]
\foreach \x in {1,1.25,1.5,1.75,2,2.25,2.5,2.75,3,3.25,3.5,3.75,4,4.25,4.5,4.75,5}
	\foreach \y in {1,1.25,1.5,1.75,2,2.25,2.5,2.75,3,3.25,3.5,3.75,4,4.25,4.5,4.75,5}
		\filldraw (\x,\y) circle (.5pt);
\foreach \i in {1,2,3,4}
	\foreach \j in {1,2,3,4}	
		\draw (\i, \j) -- (\i,\j.75) --(\i.75,\j.75) -- (\i.75,\j) -- (\i,\j);
\draw[color=blue, ultra thick] (1,1) -- (1.5,1) -- (1.5,1.75) -- (2,1.75) --(2,2.75) --(2.75,2.75) -- (3,2.75) -- (3,3.5) -- (3.75,3.5) -- (3.75,4) -- (4.5,4) -- (4.5,4.75) -- (5,4.75) -- (5,5);
\node[color=black] at (1.375,1.375) {\small{$B_0^0$}};
\node[color=black] at (2.375,2.375) {\small{$B_0^1$}};
\node[color=black] at (3.375,3.375) {\small{$B_0^2$}};
\node[color=black] at (4.375,4.375) {\small{$B_0^3$}};
\node[color=black] at (2.375,1.375) {\small{$B_1^0$}};
\node[color=black] at (3.375,2.375) {\small{$B_1^1$}};
\node[color=black] at (4.375,3.375) {\small{$B_2^1$}};
\node[color=black] at (3.375,1.375) {\small{$B_2^0$}};
\node[color=black] at (4.375,2.375) {\small{$B_2^1$}};
\node[color=black] at (4.375,1.375) {\small{$B_3^0$}};
\end{tikzpicture}

\caption{\small{A path passing through the bottom-left and top-right vertices of $B_0^j$ for each $j$.}}\label{fig:left}
\end{figure}

Denote by $\Pi_k$ the collection of paths from $(1,1)$ to $(\lfloor ns\rfloor, \lfloor nt \rfloor)$ passing through the bottom-left and top-right vertices of $B_k^j$ for each $j$. See Figure \ref{fig:left}. We have
\begin{align*}
G(\lfloor ns \rfloor, \lfloor nt \rfloor) &\geq \max_{k \leq \lfloor \sqrt{n}\rfloor} \max_{\pi \in \Pi_k} \sum_{(i,j) \in \pi} W(i,j) \geq \max_{k \leq \lfloor \sqrt{n} \rfloor} \sum_{j \leq L} G_{k,j}.
\end{align*}
It follows that
\begin{align*}
\qmeas\left(n^{-1}G(\lfloor ns \rfloor, \lfloor nt \rfloor) \leq (g(s,t) - \epsilon) \right) &\leq \qmeas \left( \max_{k \leq \lfloor \sqrt{n}\rfloor}  n^{-1} \sum_{j=0}^L G_{k,j} \leq g(s,t) - \epsilon \right) \\
&= \prod_{k=0}^{\lfloor \sqrt{n}\rfloor} \qmeas\left( n^{-1} \sum_{j=0}^L G_{k,j} \leq g(s,t) - \epsilon \right).
\end{align*}
Now, fix $\lambda > 0$ sufficiently small that $C \equiv \lambda m \frac{\epsilon}{2} - \frac{\lambda^2}{2} \aexp G_{0,0}^2 > 0$ and $\lambda \aexp G_{0,0} - \frac{\lambda^2}{2}\aexp G_{0,0}^2 < 1$ and notice that 
$\qexp\left[ e^{-\lambda G_{j,k}} \right] =  \qexp\left[ e^{-\lambda G_{0,0}} \right] \circ \tau_{(j+k)\lfloor ms \rfloor, k \lfloor mt \rfloor}.$ The ergodic theorem then implies that the following limit holds $\mu$ almost surely:
\begin{align*}
\lim_{L \to \infty}\frac{1}{L} \sum_{j=0}^L \log \qexp\left[e^{-\lambda G_{k,j}}\right] &= \E \left[ \log \qexp\left[e^{-\lambda G_{0,0}}\right] \right].
\end{align*}
Jensen's inequality gives 
$\E \left[ \log \qexp\left[e^{-\lambda G_{0,0}}\right] \right] \leq \log \aexp \left[e^{-\lambda G_{0,0}}\right] < -\lambda \aexp G_{0,0} + \frac{\lambda^2}{2} \aexp G_{0,0}^2.$ By the exponential Markov inequality and independence under $\qmeas$, we have
\begin{align*}
\frac{1}{L} \log \qmeas \left( \sum_{j=0}^L G_{k,j} < n(g(s,t) - \epsilon)\right) &\leq \frac{1}{L} \left( \sum_{j=0}^L \log \qexp \left[e^{-\lambda G_{j,k}}\right] + \lambda n (g(s,t) - \epsilon)\right).
\end{align*}
Recalling that $L^{-1}n \to m$ as $n \to \infty$, and our assumption that $\aexp G_{0,0} > m(g(s,t) - \frac{\epsilon}{2})$, it follows that
$ \limsup_{L \to \infty}  L^{-1} \log \qmeas \left( \sum_{j=0}^L G_{k,j} < n(g(s,t) - \epsilon)\right)  \leq  - \lambda m \frac{\epsilon}{2} + \frac{\lambda^2}{2} \aexp G_{0,0}^2  = - C$ almost surely.
Therefore, for each $k$ there exists a random $N_k$ so that for $n \geq N_k$
\begin{align*}
\qmeas \left( \sum_{j=0}^L G_{k,j} < n(g(s,t) - \epsilon)\right) \leq \text{exp}\left\{- n \frac{C}{2m}\right\}.
\end{align*}
For any fixed $K$ and $n \geq \max_{k \leq K} N_k$, we see that $\ameas$ almost surely we have
\begin{align*}
- \frac{1}{n} \log \qmeas\left(n^{-1}G(\lfloor ns \rfloor, \lfloor nt \rfloor) \leq (g(s,t) - \epsilon) \right) 
&\geq \sum_{k=0}^K - \frac{1}{n} \log \qmeas\left( n^{-1} \sum_{j=0}^L G_{k,j} \leq g(s,t) - \epsilon \right) \\
&\geq K \frac{C}{2m}.
\end{align*}
Sending $n \to \infty$ and then $K \to \infty$ gives the result for fixed $s,t,\epsilon >0$. For the general result, we work on the $\mu$ almost sure set where the result holds simultaneously for all rational $s,t,\epsilon >0$. Take $s,t,\epsilon > 0$ and $s_1 < s$ and $t_1 < t$ rational with the property that $\epsilon - g(s,t) + g(s_1,t_1)  >  \epsilon_1 > 0$ for rational $\epsilon_1$. This is possible by continuity of $g$. The result follows from observing that
\begin{align*}
\qmeas\left(n^{-1}G(\lfloor ns \rfloor, \lfloor nt \rfloor) \leq g(s,t) - \epsilon \right) &\leq \qmeas\left(n^{-1}G(\lfloor ns_1 \rfloor, \lfloor nt_1 \rfloor) \leq g(s_1,t_1) - \epsilon_1 \right). \qedhere
\end{align*}
\end{proof}
\begin{cor}
$\mu$ a.s. for $s,t,\lambda > 0$,
$\lim_{n \to \infty} n^{-1}\log \qexp \left[\expm\left\{-\lambda G(\lfloor ns \rfloor, \lfloor nt \rfloor)\right\}\right] = -\lambda g(s,t).$
\end{cor}

Essentially the same argument as in Lemma \ref{quenchedlb} restricted to a single diagonal $D_0$ (so that the last passage times  on $B_0^j$ are i.i.d. under $\ameas$) shows that for $r \in (0, g(s,t))$, we have
\begin{align*}
\liminf_{n \to \infty} -n^{-1}\log \ameas \left(n^{-1} G(\lf ns \rf, \lf nt \rf) \leq r \right) > 0.
\end{align*} 
To show that $n$ is the correct rate for certain left tail large deviations, we need to show that the corresponding limsup is finite for some $r \in (0, g(s,t))$. We begin by considering the natural mechanism for these deviations, which we stated previously in Section \ref{S3} as Lemma \ref{annealedub}.

\begin{proof}[Proof of Lemma \ref{annealedub}]
We may assume without loss of generality that $\{\nu_1 \in \sM^\alpha, \nu_2 \in \sM^\beta : g_{\nu_1, \nu_2}(s,t)  \in (x,y)\} \neq \emptyset$ since the right hand side is infinite otherwise. Fix a pair $\nu_1,\nu_2$ from this set and introduce the notation
\begin{align*}
A_n &= \{n^{-1}G(\lf ns \rf, \lf nt \rf) \in (x,y)\}, \qquad \frac{d\nu_1}{d\alpha} (a) = \varphi(a), \qquad \frac{d\nu_2}{d\beta}(b) = \psi(b).
\end{align*}
Since $A_n$ is measurable with respect to $\sigma\left(W(i,j) : 1 \leq i \leq \lf ns \rf, 1 \leq j \leq \lf nt \rf\right)$, we see that
\begin{align*}
\ameas_{\alpha,\beta}(A_n) = \E_{\alpha, \beta}\left[\qmeas(A_n)\right] &\geq \E_{\alpha, \beta}\left[\qmeas(A_n) \prod_{i=1}^{\lf ns \rf}1_{\{\varphi(a_i) > 0\}}\prod_{j=1}^{\lf nt \rf}1_{\{\psi(b_j) > 0\}}\right] \\
&=  \E_{\nu_1, \nu_2}\left[\qmeas(A_n) \prod_{i=1}^{\lf ns \rf}\varphi(a_i)^{-1}\prod_{j=1}^{\lf nt \rf}\psi(b_j)^{-1}\right].
\end{align*}
Taking logs and applying Jensen's inequality shows that
\begin{align*}
&- \frac{1}{n}\log \ameas_{\alpha,\beta}(A_n) \leq - \frac{1}{n}\log \E_{\nu_1, \nu_2}\left[\qmeas(A_n) \prod_{i=1}^{\lf ns \rf}\varphi(a_i)^{-1}\prod_{j=1}^{\lf nt \rf}\psi(b_j)^{-1}\right] \\
&\quad \leq \frac{1}{n \ameas_{\nu_1, \nu_2}(A_n)} \E_{\nu_1, \nu_2}\left[\qmeas(A_n) \left(\sum_{i=1}^{\lf ns \rf} \log \varphi(a_i) + \sum_{j=1}^{\lf nt \rf} \log \psi(b_j) \right)\right] - \frac{1}{n} \log \ameas_{\nu_1, \nu_2}(A_n).
\end{align*}
Note that for any measures $\nu_1,\nu_2$, we have $g_{\nu_1, \nu_2}(s,t) > 0$, so we have not divided by zero above. The last term tends to zero because $\ameas_{\nu_1, \nu_2}(A_n) \to 1$ as $n \to \infty$. For the remaining term, we note that
\begin{align*}
&\E_{\nu_1, \nu_2}\left[\qmeas(A_n) \left(\sum_{i=1}^{\lf ns \rf} \log \varphi(a_i) + \sum_{j=1}^{\lf nt \rf}\log \psi(b_j) \right)\right] = \E_{\nu_1, \nu_2}\left[\left(\sum_{i=1}^{\lf ns \rf} \log \varphi(a_i) + \sum_{j=1}^{\lf nt \rf}\log \psi(b_j) \right)\right] \\
&\qquad\qquad\qquad\qquad\qquad-\E_{\nu_1, \nu_2}\left[\qmeas(A_n^c)\log\left( \prod_{i=1}^{\lf ns \rf}\varphi(a_i) \prod_{j=1}^{\lf nt \rf}\psi(b_j)\right) \right] \\
&\qquad\qquad\qquad\qquad\qquad= \lf ns \rf \En (\nu_1 | \alpha) + \lf nt \rf \En (\nu_2 | \beta)\\
&\qquad\qquad\qquad\qquad\qquad-\E_{\alpha, \beta}\left[\qmeas(A_n^c)\prod_{i=1}^{\lf ns \rf}\varphi(a_i) \prod_{j=1}^{\lf nt \rf}\psi(b_j)\log\left( \prod_{i=1}^{\lf ns \rf}\varphi(a_i) \prod_{j=1}^{\lf nt \rf}\psi(b_j)\right) \right]
\end{align*}
But $x \log x \geq -\frac{1}{e}$ and $\qmeas(A_n^c) \in [0,1]$ so the last term is bounded above by a constant. Dividing by $n$ and taking $\limsup_{n \to \infty}$, then optimizing over $\nu_1,\nu_2$ gives the result.
\end{proof}

To show that the annealed model has non-trivial rate $n$ large deviations to the left of the shape function, it suffices to show that there exists $\nu_1 \in \sM^\alpha$ with $g_{\nu_1,\beta}(s,t) < g_{\alpha,\beta}(s,t)$. The next lemma gives mild conditions under which this is the case.

\begin{lem} \label{sufconlower}
Suppose that $\alpha$ is not degenerate and $\E^\alpha[a \log a] < \infty$. Then there exists $\nu_1$ with $\En (\nu_1 | \alpha) < \infty$ and
$g_{\nu_1,\beta}(s,t) < g_{\alpha,\beta}(s,t)$.
\end{lem}
\begin{proof}
Define $\nu_1$ by $\frac{d\nu_1}{d\alpha}(a) \simeq a.$ Note that $\En (\nu_1 | \alpha)< \infty$ by hypothesis. Let $\zeta \in [-\ubar{\alpha}, \ubar{\beta}]$ be such that $g_{\alpha,\beta}(s,t) = s \E\left[(a+\zeta)^{-1}\right] + t \E\left[(b-\zeta)^{-1}\right]$. Because $\alpha \neq \delta_c$ for any $c$, the Cauchy-Schwarz inequality gives $1= \E\left[\sqrt{a+\zeta} \sqrt{a+\zeta}^{-1}\right]^2 < \E[a+\zeta] \E\left[(a+\zeta)^{-1}\right].$ Rearranging implies that $\E\left[a(a+\zeta)^{-1}\right] < \E[a]\E\left[(a+\zeta)^{-1}\right].$ It then follows that
\begin{align*}
g_{\nu_1,\beta}(s,t) &\leq s \E[a]^{-1}\E\left[\frac{a}{a+\zeta}\right] + t \E\left[\frac{1}{b-\zeta}\right] <s \E\left[\frac{1}{a+\zeta}\right] + t \E\left[\frac{1}{b-\zeta}\right] = g_{\alpha,\beta}(s,t). \qedhere
\end{align*}
\end{proof}

\noindent We expect that the moment condition in the previous lemma is unnecessary. 

\section{Large deviation principle}
\label{S7}
We prove Theorem \ref{thm:Jcd} by working with Legendre-Fenchel transforms and appealing to convex duality.
\begin{lem}
For all $s,t >0$,
\begin{align*}
\bfJ_{s,t}^{\star}(\lambda) = \begin{cases}
\bfL_{s,t}(\lambda) & \lambda \geq 0 \\
\infty & \lambda < 0
\end{cases}, \qquad 
\bbJ_{s,t}^{\star}(\lambda) = \begin{cases}
\bbL_{s,t}(\lambda) & \lambda \geq 0 \\
\infty & \lambda < 0
\end{cases}.
\end{align*}
\end{lem}
\begin{proof}
We give the proof of the result under $\qmeas$. The proof under $\ameas$ is similar. Recall the regularity properties of $\bfJ_{s,t}(\cdot)$ proven in Proposition \ref{P3} in the appendix. The result for $\lambda < 0$ follows from the observation that $\bfJ_{s,t}(r) = 0$ for $r \leq g(s,t)$. For all $\lambda > 0$, by the exponential Markov inequality we have
\begin{align*}
\frac{1}{n} \log\qmeas\left(G(\lfloor ns \rfloor, \lfloor nt \rfloor) \geq n r\right) &\leq \frac{1}{n}\log \qexp\left[e^{\lambda G(\lfloor ns \rfloor, \lfloor nt \rfloor)}\right] - \lambda r.
\end{align*}
Sending $n \to \infty$ gives $\lambda r - \bfJ_{s,t}(r) \leq \bfL_{s,t}(\lambda)$ and taking $\sup_{r \in \mathbb{R}}$ implies $\bfJ_{s,t}^{\star}(\lambda)\leq \bfL_{s,t}(\lambda)$. For the reverse inequality, we next consider the case $\lambda \in (0, \ubar{\alpha} + \ubar{\beta})$.  Fix $M >0$ and let $\{x_i\}_{i=0}^{K}$ be a partition of $[0,M]$. We observe that
\begin{align*}
\qexp\left[e^{\lambda G(\lfloor ns \rfloor, \lfloor nt\rfloor)}\right] &= \sum_{i=1}^{K}\qexp\left[e^{\lambda G(\lfloor ns \rfloor, \lfloor nt\rfloor)} 1_{(x_{i-1}, x_i]}(n^{-1}G(\lfloor ns \rfloor, \lfloor nt\rfloor))\right]  \\
&+ \qexp\left[e^{\lambda G(\lfloor ns \rfloor, \lfloor nt\rfloor)} 1_{(M, \infty)}(n^{-1}G(\lfloor ns \rfloor, \lfloor nt\rfloor))\right].
\end{align*}
Consequently, we see that
\begin{align*}
\frac{1}{n}\log \qexp \left[e^{\lambda G(\lfloor ns \rfloor, \lfloor nt \rfloor)}\right] &\leq \max\bigg\{\max_{0 \leq i \leq K}\{ \lambda x_i + \frac{1}{n} \log \qmeas\left(n^{-1}G(\lfloor ns \rfloor, \lfloor nt \rfloor) \geq x_{i-1}\right)\},\\
&\qquad\frac{1}{n}\qexp\left[e^{\lambda G(\lfloor ns \rfloor, \lfloor nt\rfloor)} 1_{(M, \infty)}(n^{-1}G(\lfloor ns \rfloor, \lfloor nt\rfloor))\right]\bigg\} + \frac{K+1}{n}
\end{align*}
Take $\limsup_{n \to \infty}$ then $K \to \infty$. Using continuity of $r \mapsto \bfJ_{s,t}(r)$, we see that
\begin{align*}
\bfL_{s,t}(\lambda) \leq  \max_{0\leq r \leq M}\{\lambda r - \bfJ_{s,t}(r)\} \vee \limsup_{n \to \infty}\frac{1}{n} \log \qexp\left[e^{\lambda G(\lfloor ns \rfloor, \lfloor nt\rfloor)} 1_{(M, \infty)}(n^{-1}G(\lfloor ns \rfloor, \lfloor nt\rfloor))\right].
\end{align*}
Let $p,q > 1$ be such that $p^{-1} + q^{-1} = 1$ and $p \lambda < \ubar{\alpha} + \ubar{\beta}$. Then
\begin{align*}
\frac{1}{n} \log \qexp\left[e^{\lambda G(\lfloor ns \rfloor, \lfloor nt\rfloor)} 1_{(M, \infty)}(n^{-1}G(\lfloor ns \rfloor, \lfloor nt\rfloor))\right] &\leq \frac{1}{pn}\log \qexp\left[e^{\lambda p G(\lfloor ns \rfloor, \lfloor nt \rfloor)}\right] \\
&+ \frac{1}{qn} \log \qmeas\left(n^{-1} G(\lfloor ns \rfloor, \lfloor nt \rfloor \geq M\right).
\end{align*}
From this, we see that there exist deterministic constants $C_1, C_2$ such that
\begin{align*}
\limsup_{n \to \infty} \frac{1}{n} \log \qexp\left[e^{\lambda G(\lfloor ns \rfloor, \lfloor nt\rfloor)} 1_{(M, \infty)}(n^{-1}G(\lfloor ns \rfloor, \lfloor nt\rfloor))\right] \leq C_1 - C_2 \bfJ_{s,t}(M).
\end{align*}
Recall that $\lambda r \leq \bfL_{s,t}(\lambda) + \bfJ_{s,t}(r)$, so that as $M \to \infty$, $\bfJ_{s,t}(M) \to \infty$. Since $\max_{r \leq M}\{\lambda r - \bfJ_{s,t}(r)\} \leq \bfJ_{s,t}^{\star}(\lambda)$, it follows that we have $\bfL_{s,t}(\lambda) \leq \bfJ_{s,t}^{\star}(\lambda)$.

Next, we turn to the case $\lambda = \ubar{\alpha} + \ubar{\beta}$. We observe that as $\lambda \uparrow \ubar{\alpha} + \ubar{\beta}$,  $\bfL_{s,t}(\lambda) \uparrow \bfL_{s,t}(\ubar{\alpha} + \ubar{\beta}).$ Suppose that $\bfL_{s,t}(r) < \infty$. Fix $\epsilon > 0$ and take $\lambda < \ubar{\alpha} + \ubar{\beta}$ such that $\sup_{r \in \bbR}\{\lambda r - \bfJ_{s,t}(r)\} = \bfL_{s,t}(\lambda) \geq \bfL_{s,t}(\ubar{\alpha} + \ubar{\beta}) - 2\epsilon$.
Then there exists $r > 0$ so that $\lambda r - \bfJ_{s,t}(r) \geq \bfL_{s,t}(\ubar{\alpha} + \ubar{\beta}) - \epsilon$. Since $(\ubar{\alpha} + \ubar{\beta})r > \lambda r$, it follows that  $\bfJ_{s,t}^{\star}(\ubar{\alpha} + \ubar{\beta}) \geq \bfL_{s,t}(\ubar{\alpha} + \ubar{\beta}) - \epsilon$. The case $\bfL_{s,t}(\ubar{\alpha} + \ubar{\beta}) = \infty$ is similar. 

Finally, we consider the case $\lambda > \ubar{\alpha} + \ubar{\beta}$, where $\bfL_{s,t}(\lambda) = \infty$. For each $(i,j)$, we eventually have $G(\lf ns \rf, \lf nt \rf) \geq W(i,j)$. This implies that for all $(i,j)$, $\bfJ_{s,t}(r) \leq (a_i + b_j) r 1_{\{r \geq 0\}}$ and therefore $\mu$ almost surely, $\bfJ_{s,t}(r) \leq (\ubar{\alpha} + \ubar{\beta}) r1_{\{r \geq 0\}}$. Taking Legendre-Fenchel transforms of this inequality shows that $\bfJ_{s,t}^{\star}(\lambda) = \infty$.
\end{proof}
\begin{proof}[Proof of Theorem \ref{thm:Jcd}] Proposition \ref{P3} shows that $r \mapsto \bfJ_{s,t}(r)$ and $r \mapsto \bbJ_{s,t}(r)$ are real valued convex functions on $\bbR$. The result follows from taking Legendre-Fenchel transforms of the expressions in the previous lemma \cite[Theorem 12.2]{Rockafellar}.
\end{proof}

\begin{proof}[Proof of Theorem \ref{T0}]
Fix an open set $O \subset \bbR$ 
\begin{enumerate}
\item If $O \subset (-\infty, g(s,t))$ then there is nothing to prove by Lemma \ref{quenchedlb}.
\item If $g(s,t) \in O$, then 
\begin{align*}
\limsup_{n \to \infty} - n^{-1}\log \qmeas\left(n^{-1}G(\lfloor ns \rfloor, \lfloor nt \rfloor) \in O \right) = 0 = \inf_{r \in O} I_{s,t}(r)
\end{align*}
\item If $O \cap (g(s,t), \infty) \neq \emptyset$, then $O \cap (g(s,t), \infty)$ contains an interval $(r_0,r_1)$. Note that
\begin{align*}
\qmeas\left(n^{-1}G(\lfloor ns \rfloor, \lfloor nt \rfloor) \in O\right) &\geq \qmeas\left(G(\lfloor ns \rfloor, \lfloor nt \rfloor) \in (r_0,r_1)\right)\\
&= \qmeas\left(G(\lfloor ns \rfloor, \lfloor nt \rfloor) \geq r_0\right) - \qmeas\left(G(\lfloor ns \rfloor, \lfloor nt \rfloor) \geq r_1\right)
\end{align*}
Lemma \ref{L1} shows that $\bfJ_{s,t}(r)$ is strictly increasing for $r > g(s,t)$, which implies that
\begin{align*}
\limsup_{n \to \infty} - n^{-1} \log \qmeas\left(n^{-1}G(\lfloor ns \rfloor, \lfloor nt \rfloor) \in O\right) \leq \bfJ_{s,t}(r_0).
\end{align*}
Let $r_n \in O \cap (g(s,t),\infty)$ be a sequence with $r_n \downarrow r_\infty = \inf\{x : x \in O \cap (g(s,t),\infty)\}$. Then because $\bfJ_{s,t}(r)$ is continuous and non-decreasing, we see that
\begin{align*}
\limsup_{n \to \infty} - n^{-1} \log \qmeas\left(n^{-1}G(\lfloor ns \rfloor, \lfloor nt \rfloor) \in O\right) \leq \bfJ_{s,t}(r_\infty) = \inf_{r \in O \cap (g(s,t),\infty)} \bfI_{s,t}(r) = \inf_{r \in O} \bfI_{s,t}(r).
\end{align*}
\end{enumerate}
The upper bound follows from the regularity of $\bfJ_{s,t}$, Theorem \ref{thm:Jcd} and Lemma \ref{quenchedlb}.
\end{proof}

\section{Relative entropy and the rate functions} \label{S8}
We now turn to the proof of Theorem \ref{thm:annealedent}. Our argument proving this result is purely convex analytic and does not show the probabilistic interpretation mentioned before the statement of the theorem. We begin with a technical lemma.
\begin{lem}
For $r > 0$, the map $(\alpha,\beta) \mapsto \bfI_{s,t}^{\alpha,\beta}(r)$ is convex on $\sM_1(\bbR_+)^2$.
\end{lem}
\begin{proof}
Using (\ref{E28}), one can check that $(\alpha,\beta) \mapsto g_{\alpha,\beta}(s,t)$ is concave on $\sM(\bbR_+)^2$. Thus, $\{(\alpha,\beta) : g_{\alpha,\beta}(s,t) \geq r\}$ is convex. Define for $(\alpha,\beta) \in \sM_1(\bbR_+)^2$
\begin{align*}
F(\alpha,\beta) = \sup \limits_{\substack{\lambda \in (0, \ubar{\alpha} + \ubar{\beta}) \\ z \in(-\ubar{\alpha},\ubar{\beta}-\lambda)}} \left\{\lambda r - s \E^{\alpha} \left[\log\left(\frac{a+z+\lambda}{a+z}\right)\right] - t\E^{\beta} \left[\log\left(\frac{b-z}{b-z-\lambda}\right)\right] \right\}.
\end{align*}
Fix $\alpha_1,\alpha_2,\beta_1,\beta_2 \in \sM_1(\bbR_+)$ and $\delta \in (0,1)$. Denote by $\alpha^\delta = \delta \alpha_1 + (1-\delta)\alpha_2$ and by $\beta^\delta = \delta \beta_1 + (1-\delta)\beta_2$. Note that $\ubar{\alpha}^\delta = \ubar{\alpha}_1 \wedge \ubar{\alpha}_2$ and $\ubar{\beta}^\delta = \ubar{\beta}_1 \wedge \ubar{\beta}_2$. Then
\begin{align*}
F(\alpha^\delta,\beta^\delta) &= \sup \limits_{\substack{\lambda \in (0, \ubar{\alpha}^\delta + \ubar{\beta}^\delta) \\ z \in(-\ubar{\alpha}^\delta,\ubar{\beta}^\delta-\lambda)}} \left\{\lambda r - s \E^{\alpha^\delta} \left[\log\left(\frac{a+z+\lambda}{a+z}\right)\right] - t\E^{\beta^\delta} \left[\log\left(\frac{b-z}{b-z-\lambda}\right)\right] \right\} \\
&\leq  \delta \sup \limits_{\substack{\lambda \in (0, \ubar{\alpha}^\delta + \ubar{\beta}^\delta) \\ z \in(-\ubar{\alpha}^\delta,\ubar{\beta}^\delta-\lambda)}} \left\{\lambda r - s \E^{\alpha_1} \left[\log\left(\frac{a+z+\lambda}{a+z}\right)\right] - t\E^{\beta_1} \left[\log\left(\frac{b-z}{b-z-\lambda}\right)\right] \right\} \\
&+ (1-\delta) \sup \limits_{\substack{\lambda \in (0, \ubar{\alpha}^\delta + \ubar{\beta}^\delta) \\ z \in(-\ubar{\alpha}^\delta,\ubar{\beta}^\delta-\lambda)}} \left\{\lambda r - s \E^{\alpha_2} \left[\log\left(\frac{a+z+\lambda}{a+z}\right)\right] - t\E^{\beta_2} \left[\log\left(\frac{b-z}{b-z-\lambda}\right)\right] \right\} \\
&\leq  \delta \sup \limits_{\substack{\lambda \in (0, \ubar{\alpha}_1 + \ubar{\beta}_1) \\ z \in(-\ubar{\alpha}_1,\ubar{\beta}_1-\lambda)}} \left\{\lambda r - s \E^{\alpha_1} \left[\log\left(\frac{a+z+\lambda}{a+z}\right)\right] - t\E^{\beta_1} \left[\log\left(\frac{b-z}{b-z-\lambda}\right)\right] \right\} \\
&+(1-\delta) \sup \limits_{\substack{\lambda \in (0, \ubar{\alpha}_2 + \ubar{\beta}_2) \\ z \in(-\ubar{\alpha}_2,\ubar{\beta}_2-\lambda)}} \left\{\lambda r - s \E^{\alpha_2} \left[\log\left(\frac{a+z+\lambda}{a+z}\right)\right] - t\E^{\beta_2} \left[\log\left(\frac{b-z}{b-z-\lambda}\right)\right] \right\},
\end{align*}
so that $F$ is convex on $\sM(\bbR^+)^2$. Then we see from (\ref{Istdef}) that $(\alpha,\beta) \mapsto \bfI_{s,t}^{\alpha,\beta}(r)$ is convex on $\sM(\bbR^+)^2$. 
\end{proof}

\begin{proof}[Proof of Theorem \ref{thm:annealedent}]
Theorem \ref{thm:Jcd} and the variational characterization of relative entropy, \cite[Theorem 5.4]{Rassoul-AghaSeppalainen}, imply that for $r > g(s,t)$,
\begin{align*}
\bbJ_{s,t}^{\alpha,\beta}(r)= &\sup \limits_{\substack{\lambda \in (0, \ubar{\alpha} + \ubar{\beta}) \\ z \in (-\ubar{\alpha}, \ubar{\beta} - \lambda)}}\left\{\lambda r - s \log \E^\alpha \left[\frac{a+z+\lambda}{a+z}\right] - t\log\E^\beta \left[\frac{b-z}{b-z-\lambda}\right]\right\} \\
=&\sup \limits_{\substack{\lambda \in (0, \ubar{\alpha} + \ubar{\beta}) \\ z \in (-\ubar{\alpha}, \ubar{\beta} - \lambda)}}\inf \limits_{\substack{\nu_1 \in \sM^\alpha \\ \nu_2 \in \sM^\beta}}\bigg\{\lambda r - s  \E^{\nu_1} \left[\log\left(\frac{a+z+\lambda}{a+z}\right)\right] - t\E^{\nu_2} \left[\log\left(\frac{b-z}{b-z-\lambda}\right)\right] \\
&\qquad\qquad\qquad\qquad\qquad+ s\En (\nu_1 | \alpha) + t \En (\nu_2 | \beta)\bigg\}\\
&\leq \inf \limits_{\substack{\nu_1 \in \sM^\alpha \\ \nu_2 \in \sM^\beta}}\sup \limits_{\substack{\lambda \in (0, \ubar{\alpha} + \ubar{\beta}) \\ z \in (-\ubar{\alpha}, \ubar{\beta} - \lambda)}}\bigg\{\lambda r - s  \E^{\nu_1} \left[\log\left(\frac{a+z+\lambda}{a+z}\right)\right] - t\E^{\nu_2} \left[\log\left(\frac{b-z}{b-z-\lambda}\right)\right] \\
&\qquad\qquad\qquad\qquad\qquad+ s\En (\nu_1 | \alpha) + t \En (\nu_2 | \beta)\bigg\}.
\end{align*}
Note that if $\nu_1 \ll \alpha$, it must be the case that $\ubar{\nu}_1 \geq \ubar{\alpha}$ and similarly, $\ubar{\nu}_2 \geq \ubar{\beta}$. It follows that we may extend the region in the inner supremum to obtain
\begin{align*}
\bbJ_{s,t}^{\alpha,\beta}(r) &\leq \inf_{\nu_1, \nu_2}\left\{\bfI_{s,t}^{\nu_1,\nu_2}(r) + s\En (\nu_1 | \alpha) + t \En (\nu_2 | \beta) \right\}.
\end{align*}
The map $(\nu_1,\nu_2) \mapsto \bfI_{s,t}^{\nu_1,\nu_2}(r) + s\En (\nu_1 | \alpha) + t \En (\nu_2 | \beta)$ is strictly convex on the convex set $\sM^\alpha \times \sM^\beta$ so at most one minimizing pair $(\nu_1,\nu_2)$ exists. It therefore suffices to show that  we have equality with the measures $\nu_1, \nu_2$ defined in the statement of the theorem. We argue this by cases.

A maximizing pair $\lam, \zs$ satisfying $\lam \in [0, \ubar{\alpha} + \ubar{\beta}], \zs \in [-\ubar{\alpha}, \ubar{\beta} - \lam]$ exist for the annealed right-tail rate function by Lemmas \ref{L9} and \ref{L1}. ($\zs$ denotes $\zs(\lam)$ in the notation of Section \ref{Sreg}.  Also, by Corollary \ref{C1}, $\zs(\ubar{\alpha}+\ubar{\beta}) = -\ubar{\alpha}$). Note that $\lam = 0$ is impossible because $\bbJ_{s,t}^{\alpha,\beta}(r) > 0$ by Lemma \ref{L1}. If $\lam \in (0, \ubar{\alpha} + \ubar{\beta})$ and $\zs \in (-\ubar{\alpha}, \ubar{\beta} - \lam)$, then $\nu_1\in \sM^\alpha$ and $\nu_2 \in \sM^\beta$ because their densities with respect to $\alpha$ and $\beta$ are bounded. Taking derivatives in (\ref{maxpair}), we see that $\zs$ and $\lam$ solve
\begin{align}
0 &= s \E^{\nu_1}\left[\frac{1}{a+\zs} - \frac{1}{a+\zs+\lam}\right] + t \E^{\nu_2}\left[\frac{1}{b-\zs} - \frac{1}{b-\zs-\lam}\right] \label{E148}\\
0 &= r - s\E^{\nu_1}\left[\frac{1}{a+\zs+\lam}\right] - t \E^{\nu_2}\left[\frac{1}{b-\zs-\lam}\right].\label{E149}
\end{align}
These are precisely the first order conditions implying that
\begin{align*}
\bfI_{s,t}^{\nu_1,\nu_2}(r) &= \lam r - s \E^{\nu_1}\left[\log \frac{a+\zs+\lam}{a+\zs}\right] - t \E^{\nu_2}\left[\log \frac{b-\zs}{b-\zs-\lam}\right].
\end{align*}
The definition of relative entropy and a little algebra then show that
\begin{align*}
\bbJ_{s,t}^{\alpha,\beta}(r) &= \bfI_{s,t}^{\nu_1,\nu_2}(r) + s\En (\nu_1|\alpha) + t \En (\nu_2|\alpha).
\end{align*}
The remaining cases are similar in that once we know that the extremizers are the same for $\bbJ_{s,t}^{\alpha,\beta}(r)$ and $\bfI_{s,t}^{\nu_1,\nu_2}(r)$, the result follows. The necessary and sufficient conditions in Lemmas \ref{L3} and \ref{L9} show that $\nu_1$ and $\nu_2$ are well defined and that this equality continues to hold if $\lam < \ubar{\alpha} + \ubar{\beta} $ and $\zs = -\ubar{\alpha}$ or $\zs = \ubar{\beta} - \lam$. The only remaining case is $\lam = \ubar{\alpha} + \ubar{\beta}$ and $\zs = -\ubar{\alpha}$. $\lam = \ubar{\alpha} + \ubar{\beta}$ is equivalent to $r \geq (\bbL_{s,t}^{\alpha,\beta})'(\ubar{\alpha}+\ubar{\beta})$. By Corollary \ref{L10}, this condition implies that $\nu_1$ and $\nu_2$ are well defined and $(\bfL_{s,t}^{\nu_1,\nu_2})'(\ubar{\alpha}+\ubar{\beta}) =  (\bbL_{s,t}^{\alpha,\beta})'(\ubar{\alpha}+\ubar{\beta})$. The result follows.
\end{proof}
\section{Scaling estimates}
\label{S10}

In this section, we prove the scaling estimates for the quenched and the annealed rate functions. See the discussion  Section \ref{Sreg} for the notation below. If $c_1 < s/t < c_2$ we have $\partial_z g_\zeta(s, t) = 0$ and, therefore,   
\begin{align}
\label{E79}
g_z(s, t) = g(s, t) + \partial_{z}^2 g_{\zeta}(s, t)(z-\zeta)^2/2+o((z-\zeta)^2). 
\end{align}
In fact, (\ref{E79}) holds for $s/t = c_1$ and $s/t = c_2$ as well provided that 
\begin{align}
\E\left[\frac{1}{(a-\ubar{\alpha})^3}\right] < \infty,  \qquad \qquad\E\left[\frac{1}{(b-\ubar{\beta})^3}\right] < \infty; \label{E81.2}
\end{align}
that is, assuming that $\partial_z^2 g_{z}(s, t)$ has limits at the endpoints $-\ubar{\alpha}$ and $\ubar{\beta}$. 

\begin{proof}[Proof of Theorem \ref{T3}]
For $\epsilon > 0$ sufficiently small, we have 
\begin{align}
\bfI_{s, t}'(r) = \lam(r), \qquad
\bfL_{s, t}'(\lam(r)) = r \label{E81}
\end{align}
whenever $g(s, t) \le r \le g(s, t) + \epsilon$ by Lemma \ref{L1}. We begin with the case $c_1 < s/t < c_2$. Then $\zeta \in (-\ubar{\alpha}, \ubar{\beta})$. We recall $\lambda_1$ and $\lambda_2$ defined in Lemma \ref{L3}. Because $\partial_z F(-\ubar{\alpha}, 0) = \partial_z g_{-\ubar{\alpha}}(s, t) < 0$ and $\partial_z F(-\ubar{\alpha}, 0) = \partial_z g_{-\ubar{\alpha}}(s, t) > 0$, we conclude that $\lambda_1 > 0$ and $\lambda_2 > 0$. Hence, 
\begin{align*}
\zs'(\lambda) &= -\frac{\partial_{\lambda} \partial_z F(\zs, \lambda)}{\partial_z^2 F(\zs, \lambda)}
= -\frac{s\E\left[\dfrac{1}{(a+\zs)(a+\zs+\lambda)^2}\right]+t\E\left[\dfrac{1}{(b-\zs)(b-\zs-\lambda)^2}\right]}{s\E\left[\dfrac{2a+2\zs+\lambda}{(a+\zs+\lambda)^2(a+\zs)^2}\right]+t\E\left[\dfrac{2b-2\zs-\lambda}{(b-\zs-\lambda)^2(b-\zs)^2}\right]}. 
\end{align*}
for $0 < \lambda < \lambda_1 \wedge \lambda_2$. Letting $\lambda \downarrow 0$ yields 
$\zs'(0^+) = -1/2$.  
It follows that $\zs(\lambda) = \zeta -\lambda/2 + o(\lambda)$ as $\lambda \downarrow 0$. We obtain  $\bfL_{s, t}'(\lambda) = g_{\zs+\lambda}(s, t) = g(s, t)  + \partial_{z}^2 g_{\zeta}(s, t) \lambda^2/8 + o(\lambda^2)$
as $\lambda \downarrow 0$. Then,
\begin{align*}
\bfI_{s, t}'(g(s, t)+\epsilon) = \frac{2\sqrt{2}}{\sqrt{\partial_z^2 g_{\zeta}(s, t)}} \epsilon^{1/2} + o(\epsilon^{1/2}), 
\end{align*}
and integrating gives 
\begin{align}
\bfI_{s, t}(g(s, t)+\epsilon) &= \frac{4\sqrt{2} \epsilon^{3/2}}{3\sqrt{\partial_z^2 g_{\zeta}(s, t)}} + o(\epsilon^{3/2}) = \frac{4}{3}\frac{\epsilon^{3/2}}{\sqrt{s\E\left[\dfrac{1}{(a+\zeta)^3}\right] + t\E\left[\dfrac{1}{(b-\zeta)^3}\right]}} + o(\epsilon^{3/2}) \label{E83}
\end{align}
as $\epsilon \downarrow 0$.  Now, suppose that $s/t \le c_1$. Then $E[(a-\ubar{\alpha})^{-2}] < \infty$, $\zeta = -\ubar{\alpha}$ and $\zs = -\ubar{\alpha}$. Under condition (\ref{E81.2}), when $c_1 = s/t$,
$\bfL_{s, t}'(\lambda) = g_{-\ubar{\alpha}+\lambda}(s, t) = g(s, t)  + \partial_{z}^2 g_{-\ubar{\alpha}}(s, t) \lambda^2/2 + o(\lambda^2)$
and we reach (\ref{E83}) multiplied with $1/2$. If $c_1 > s/t$ then $\partial_z g_{-\ubar{\alpha}}(s, t) > 0$ and we have $
\bfL_{s, t}'(\lambda) = g_{-\ubar{\alpha}+\lambda}(s, t) = g(s, t) + \partial_{z} g_{-\ubar{\alpha}}(s, t) \lambda + o(\lambda).$  This leads to 
\begin{align*}
\bfI_{s, t}(g(s, t)+\epsilon) &= \frac{\epsilon^{2}}{2\partial_z g_{-\ubar{\alpha}}(s, t)} + o(\epsilon^{2}) = \frac{1}{2} \frac{\epsilon^{2}}{-s\E\left[\dfrac{1}{(a-\ubar{\alpha})^2}\right] + t\E\left[\dfrac{1}{(b+\ubar{\alpha})^2}\right]}  + o(\epsilon^{2}).
\end{align*}
Analysis of the case  $s/t \ge c_2$ is similar.
\end{proof}

\begin{proof}[Proof of Theorem \ref{T2}]
In the case $c_1 < s/t < c_2$, H\"older's inequality gives
\begin{align*}
\lim_{\lambda \downarrow 0} \zs'(\lambda) &= -\lim_{\lambda \downarrow 0} \frac{\partial_{\lambda}\partial_z F(\zs, \lambda)}{\partial_z^2 F(\zs, \lambda)} \\
&= -\frac{s\E\left[\dfrac{1}{(a+\zeta)^2}\right]\E\left[\dfrac{1}{a+\zeta}\right]+2t\E\left[\dfrac{1}{(b-\zeta)^3}\right]-t\E\left[\dfrac{1}{(b-\zeta)^2}\right]\E\left[\dfrac{1}{b-\zeta}\right]}{2s\E\left[\dfrac{1}{(a+\zeta)^3}\right] + 2t\E\left[\dfrac{1}{(b-\zeta)^3}\right]}\\
&\le -\frac{s\E\left[\dfrac{1}{(a+\zeta)^2}\right]\E\left[\dfrac{1}{a+\zeta}\right]+t\E\left[\dfrac{1}{(b-\zeta)^3}\right]}{2s\E\left[\dfrac{1}{(a+\zeta)^3}\right] + 2t\E\left[\dfrac{1}{(b-\zeta)^3}\right]}.
\end{align*}
Hence, $\zs(\lambda) = \zeta + c\lambda + o(\lambda),$ where $c < 0$. We have 
\begin{align}
\bbL'_{s, t}(\lambda) &= s \frac{\E \left[\dfrac{1}{a+\zs}\right]}{\E \left[\dfrac{a+\zs+\lambda}{a+\zs}\right]} + t \frac{\E \left[\dfrac{b-\zs}{(b-\zs-\lambda)^2}\right]}{\E \left[\dfrac{b-\zs}{b-\zs-\lambda}\right]} \label{E85} \\
&= g_{\zs+\lambda}(s, t) + \lambda \left(s\Var\left[\frac{1}{a+\zeta}\right] + t\Var\left[\frac{1}{b-\zeta}\right]\right) + o(\lambda) \label{E86}\\
&= g(s, t) + \lambda \left(s\Var\left[\frac{1}{a+\zeta}\right] + t\Var\left[\frac{1}{b-\zeta}\right]\right) + o(\lambda) \label{E87}. 
\end{align}
Then, arguing as in the preceding proof, we obtain 
\begin{align}
\bbJ_{s, t}(g(s, t) + \epsilon) = \frac{1}{2}\frac{\epsilon^2}{s\Var\left[\dfrac{1}{a+\zeta}\right] + t\Var\left[\dfrac{1}{b-\zeta}\right]}+o(\epsilon) \label{E88}. 
\end{align}

Now consider $s/t \le c_1$. Then $\zeta = \zs = -\ubar{\alpha}$ and (\ref{E85}) still holds. If $s/t = c_1$  subsequent arguments go through assuming (\ref{E81.2}). This condition is needed in step (\ref{E86}), which relies on (\ref{E79}) with $\zeta = -\ubar{\alpha}$. Hence, we have (\ref{E88}). If $s/t < c_1$ then the coefficient of $\lambda$ in (\ref{E87}) has an additional term $\partial_z g_{-\ubar{\alpha}}(s, t) > 0$, which leads to  
\begin{align*}
\bbJ_{s, t}(g(s, t) + \epsilon) = \frac{1}{2}\frac{\epsilon^2}{-s\E\left[\dfrac{1}{a-\ubar{\alpha}}\right]^2 + t\Var\left[\dfrac{1}{b+\ubar{\alpha}}\right]+t\E\left[\dfrac{1}{(b+\ubar{\alpha})^2}\right]}+o(\epsilon).
\end{align*} 

The case $s/t \ge c_2$ is analyzed similarly.
\end{proof}

\appendix

\section{Right tail rate functions and Lyapunov exponents} \label{A}

\begin{prop}
\label{P3}
\begin{enumerate}[(a)]
\item $\mu$-a.s., for $s, t > 0$ and $r \in \bbR$, there exists (nonrandom) $\bfJ_{s, t}(r) \in [0, \infty)$ such that 
\begin{equation}
\label{E48}
\lim \limits_{n \rightarrow \infty} -\frac{1}{n} \log \bfP_{\bfa, \bfb}(G(\lf ns \rf, \lf nt \rf) \ge nr) = \bfJ_{s, t}(r). 
\end{equation}
\item For all $s, t > 0$ and $r \in \bbR$, there exists $\bbJ_{s, t}(r) \in [0, \infty)$ such that 
\begin{equation}
\label{E48.1}
\lim \limits_{n \rightarrow \infty} -\frac{1}{n} \log \bbP(G(\lf ns \rf, \lf nt \rf) \ge nr) = \bbJ_{s, t}(r). 
\end{equation}
\item $\bfJ$ and $\bbJ$ are convex and homogeneous in $(s, t, r)$, nonincreasing in $(s, t)$ and nondecreasing in $r$.
\end{enumerate}
\end{prop}
\begin{proof}
Fix $r \in \bbR$ and $s, t \in \bbN$. For integers $0 \le m < n$, define 
\[X_{m, n} = -\log \bfP_{\tau_{ms}(\bfa), \tau_{mt}(\bfb)}(G((n-m)s, (n-m)t) \ge (n-m)r).\]
We verify that $\{X_{m, n}\}$ satisfy the hypotheses of the subadditive ergodic theorem in \cite{Liggett}. For subadditiviy, note that 
\begin{align*}
X_{0, n} &= -\log \bfP_{\bfa, \bfb}(G(ns, nt) \ge nr) \\
&\le -\log \bfP_{\bfa, \bfb}(G(ms, mt) \ge mr) - \log \bfP_{\bfa, \bfb}(G((n-m)s, (n-m)t) \circ \theta_{ms, mt} \ge (n-m)r)\\
&= X_{0, m} + X_{m, n}. 
\end{align*}
For $k \in \bbN$, by the ergodicity assumptions on $\mu$, the sequence $(X_{k, k+n})_{n \in \bbN}$ has the same distribution as $(X_{0, n})_{n \in \bbN}$ and the sequence 
$(X_{(n-1)k, nk})_{n \in \bbN}$ is ergodic. Moreover, $X_{0, n} \ge 0$ and 
\begin{align}
\E X_{0, n} &\le \E \left[-\log \bfP_{\bfa, \bfb}(W(1, 1) \ge nr)\right] = n \max\{r, 0\} \E[a + b] < \infty. \label{E51}  
\end{align}  
Hence, by the subadditive ergodic theorem, (\ref{E48}) holds $\mu$-a.s. (and in expectation under $\mu$) with 
\begin{align}
\bfJ_{s, t}(r) &= \lim \limits_{n \rightarrow \infty} \frac{1}{n} \E X_{0, n} = \lim \limits_{n \rightarrow \infty}-\frac{1}{n} \E \log \bfP_{\bfa, \bfb}(G(ns, nt) \ge nr). \label{E50}
\end{align}

We record some properties of $\bfJ_{s, t}(r)$ for $s, t \in \bbN$ and $r \in \bbR$. It is clear from (\ref{E50}) that $\bfJ_{s, t}(r)$ is nonincreasing in $(s, t)$ and nondecreasing in $r$. In addition, $\bfJ_{s, t}(r) = 0$ for $r \le 0$ as $G$ is nonnegative, and $\bfJ_{cs, ct}(cr) = c\bfJ_{s, t}(r)$
for $c \in \bbN$. 
By (\ref{E51}), 
$\bfJ_{s, t}(r) \le r\E[a+b] < \infty$ for $r \ge 0$. Also, for $s_1, s_2, t_1, t_2 \in \bbN$ and $r_1, r_2 \in \bbR$, we have 
\begin{align*}\E \log \bfP_{\bfa, \bfb}(G(n(s_1+s_2), n(t_1+t_2)) \ge n(r_1+r_2)) \ge &\E \log \bfP_{\bfa, \bfb}(G(ns_1, nt_1) \ge nr_1)\\
\cdot &\E \log \bfP_{\bfa, \bfb}(G(ns_2, nt_2) \ge nr_2)
\end{align*}
for $n \in \bbN$, which gives 
$\bfJ_{s_1+s_2, t_1+t_2}(r_1+r_2) \ge \bfJ_{s_1, t_1}(r_1) + \bfJ_{s_2, t_2}(r_2)$.
Then, for $0 \le r \le r' \le r+\frac{1}{n}$,  
\begin{align}
\bfJ_{s, t}(r')-\bfJ_{s, t}(r) &\le \bfJ_{s, t}(r+1/n)-\bfJ_{s, t}(r) \nonumber \\
&= \frac{1}{n+1}\bfJ_{(n+1)s, (n+1)t}(nr+r+1+1/n)-\bfJ_{s, t}(r) \nonumber  \\
&\le \frac{\bfJ_{ns, nt}(nr)}{n+1} - \bfJ_{s, t}(r) + \frac{\bfJ_{s, t}(r+2)}{n+1} \nonumber  \\
&= \frac{\bfJ_{s, t}(r+2)-\bfJ_{s, t}(r)}{n+1} \nonumber \\
&\le \frac{2r+2}{n}\E[a+b] \label{E53}, 
\end{align}
which shows continuity of $\bfJ_{s, t}(r)$ in $r$.  

There exists a $\mu$-a.s. event $E$ on which (\ref{E48}) holds for all $s, t \in \bbN$ and $r \in \bbQ$. It follows from the monotonicity of $\log \bfP_{\bfa, \bfb}(G(ns, nt) \ge nr)$ in $r$ and continuity of $\bfJ_{s, t}$ that (\ref{E48}) holds for all $s, t \in \bbN$ and $r \in \bbR$ on $E$. From now on, let us work with $(\bfa, \bfb) \in E$. 

For $c > 0$, $\delta \in (0, 1)$ and large enough $n \in \bbN$, we have 
\begin{equation}
\label{E54}
\begin{aligned}
-\log \bfP_{\bfa, \bfb}(G(\lf ncs \rf, \lf nct \rf) \ge nr) &\le -\log \bfP_{\bfa, \bfb}(G(\lf cn \rf s, \lf cn \rf t) \ge \lf cn \rf r(1+\delta)) \\
-\log \bfP_{\bfa, \bfb}(G(\lf ncs \rf, \lf nct \rf) \ge nr) &\ge -\log \bfP_{\bfa, \bfb}(G(\lc cn \rc s, \lc cn \rc t) \ge \lc cn \rc r(1-\delta)). 
\end{aligned}
\end{equation}
It follows from these inequalities and continuity of $\bfJ_{s, t}$ that (\ref{E48}) holds on $E$ with $\bfJ_{cs, ct}(cr) = c\bfJ_{s, t}(r)$. In particular, $\bfJ_{s, t}(r)$ exists for rational $s, t > 0$. Moreover, by homogeneity, the properties of $\bfJ_{s, t}(r)$ noted in preceding paragraph hold for rational $s, t > 0$ as well. 

For $s, t, \delta > 0$, choose rational $s', t'$ such that $\frac{s'}{1+\delta} < s \le s'$ and $\frac{t'}{1+\delta} < t \le t'$. Then 
\begin{equation}
\label{E55}
\begin{aligned}
-\log \bfP_{\bfa, \bfb}(G(\lf ns \rf, \lf nt \rf) \ge nr) &\ge -\log \bfP_{\bfa, \bfb}(G(\lf ns' \rf, \lf nt' \rf) \ge nr) \\
-\log \bfP_{\bfa, \bfb}(G(\lf ns \rf, \lf nt \rf) \ge nr) &\le -\log \bfP_{\bfa, \bfb}(G(\lf ns'/(1+\delta) \rf, \lf nt'/(1+\delta) \rf ) \ge nr).
\end{aligned}
\end{equation}
It follows that 
\begin{align*}
\liminf \limits_{n \rightarrow \infty} -\frac{1}{n} \log \bfP_{\bfa, \bfb}(G(\lf ns \rf, \lf nt \rf) \ge nr) &\ge \bfJ_{s', t'}(r) \\
\limsup \limits_{n \rightarrow \infty} -\frac{1}{n} \log \bfP_{\bfa, \bfb}(G(\lf ns \rf, \lf nt \rf) \ge nr) &\le \bfJ_{s'/(1+\delta), t'/(1+\delta)}(r) = \bfJ_{s', t'}((1+\delta)r)/(1+\delta). 
\end{align*}
Using (\ref{E53}), we obtain 
\begin{align*}
\frac{\bfJ_{s', t'}((1+\delta)r)}{1+\delta}-\bfJ_{s', t'}(r) &\le \bfJ_{s', t'}((1+\delta)r)-\bfJ_{s', t'}(r) \le \frac{2r+2}{\lc (r\delta)^{-1}\rc} \E[a+b]. 
\end{align*}
As $\delta \downarrow 0$, we have $s' \downarrow s$ and $t' \downarrow t$. Hence, we conclude that $\bfJ_{s, t}(r)$ exists and equals the limit of $\bfJ_{s', t'}(r)$, and also enjoys the properties of mentioned above. Finally, it follows from subadditivity and homogeneity that $\bfJ$ is convex. 
\end{proof}

\begin{prop}
\label{P1}\ 

\begin{enumerate}[(a)]
\item $\mu$-a.s.,  for any $s, t > 0$ and $\lambda \in \bbR$, there exists $\bfL_{s, t}(\lambda) \in [-\infty, \infty]$ such that,   
\begin{equation}\label{E25}
\lim \limits_{n \rightarrow \infty} \frac{1}{n} \log \bfE_{\bfa, \bfb}[e^{\lambda G(\lf ns \rf, \lf nt \rf)}] = \bfL_{s, t}(\lambda)
\end{equation}
\item For any $s, t > 0$ and $\lambda \in \bbR$, 
\begin{equation}\label{E25.1}
\lim \limits_{n \rightarrow \infty} \frac{1}{n} \log \bbE[e^{\lambda G(\lf ns \rf, \lf nt \rf)}] = \bbL_{s, t}(\lambda) 
\end{equation}
\item $\bfL_{s, t}(\lambda)$ and $\bbL_{s, t}(\lambda)$ are nondecreasing and convex in $\lambda$. 
\item $\lambda \bfL_{s, t}(\lambda)$ and $\lambda \bbL_{s, t}(\lambda)$ are nondecreasing, homogeneous and concave in $(s, t)$.  
\end{enumerate}
\end{prop}
\begin{proof}
Fix $\lambda \in \bbR$ and $s, t \in \bbN$. Define 
\[X_{m, n} = -\lambda \log \bfE_{\tau_{ms}(\bfa), \tau_{mt}(\bfb)}\left[e^{\lambda G((n-m)s, (n-m)t)}\right]\]
for integers $0 \le m < n$. Then $\{X_{m, n}: 0 \le m < n\}$ are nonpositive and subadditive, and the conditions of the subadditive ergodic theorem are in place to claim the existence of $\bfL_{s, t}(\lambda) \in [-\infty, \infty]$ such that (\ref{E25}) holds $\mu$-a.s. 

For $\lambda \in \bbR, s, t \in \bbN$ and $c > 0$, we have 
\begin{align*}
-\lambda \log \bfE_{\bfa, \bfb}\left[e^{\lambda G(\lc nc \rc s, \lc nc \rc t)}\right] \le -\lambda \log \bfE_{\bfa, \bfb}\left[e^{\lambda G(\lf ncs \rf, \lf nct \rf)}\right] &\le -\lambda \log \bfE_{\bfa, \bfb}\left[e^{\lambda G(\lf nc \rf s, \lf nc \rf t)}\right] 
\end{align*}
Also, for $\lambda \in \bbR$, $s, s', t, t', \delta > 0$ such that $s', t'$ are rational, $\frac{s'}{1+\delta} < s \le s'$ and $\frac{t'}{1+\delta} < t \le t'$, 
\begin{align*}
-\lambda \log \bfE_{\bfa, \bfb}\left[e^{\lambda G(\lf ns' \rf, \lf nt' \rf)}\right] \le -\lambda \log \bfE_{\bfa, \bfb}\left[e^{\lambda G(\lf ns \rf, \lf nt \rf)}\right] &\le -\lambda \log \bfE_{\bfa, \bfb}\left[e^{\lambda G(\lf \frac{ns'}{1+\delta}\rf, \lf \frac{nt'}{1+\delta}\rf)} \right].
\end{align*}
Using these inequalities as in the preceding proof, we obtain (\ref{E25}) for all $s,t > 0$ $\mu$-a.s. and the claimed properties of the function $(s, t) \mapsto \lambda \bfL_{s, t}(\lambda)$. 

Now fix $s, t > 0$. Note that $\bfL_{s, t}(\lambda)$ is nondecreasing in $\lambda$. Let $\lambda_0 = \sup_{\lambda \in \bbR} \{\bfL_{s, t}(\lambda) < \infty\}$. For $\lambda_1, \lambda_2 \in \bbR$ and $c_1, c_2 \in (0, 1)$ with $c_1 + c_2 = 1$, by H\"{o}lder's inequality,  
\[\log \bfE_{\bfa, \bfb}\left[e^{(c_1 \lambda_1 + c_2 \lambda_2)G(\lf ns \rf, \lf nt \rf)}\right] \le c_1 \log \bfE_{\bfa, \bfb}\left[e^{\lambda_1G(\lf ns \rf, \lf nt \rf)}\right] + c_2 \log \bfE_{\bfa, \bfb}\left[e^{\lambda_2G(\lf ns \rf, \lf nt \rf)}\right], \]
which implies that $\bfL_{s, t}(c_1 \lambda_1 + c_2 \lambda_2) \le c_1 \bfL_{s, t}(\lambda_1) + c_2 \bfL_{s, t}(\lambda_2)$. Hence, $\bfL_{s, t}(\lambda)$ is continuous in $\lambda$ on $(-\infty, \lambda_0)$.  Using this and the monotonicity of last-passage times, we deduce that (\ref{E25}) holds for all $s, t > 0$ and $\lambda \in \bbR$ $\mu$-a.s. 
\end{proof}

\bibliographystyle{habbrv}
\bibliography{CGM_LDP}

\begin{thebibliography}{10}

\bibitem{AGZ10}
G.~W. Anderson, A.~Guionnet, and O.~Zeitouni.
\newblock {\em An introduction to random matrices}, volume 118 of {\em
  Cambridge Studies in Advanced Mathematics}.
\newblock Cambridge University Press, Cambridge, 2010.

\bibitem{BorodinCorwinFerrari}
A.~Borodin, I.~Corwin, and P.~Ferrari.
\newblock Free energy fluctuations for directed polymers in random media in
  {$1+1$} dimension.
\newblock {\em Comm. Pure Appl. Math.}, 67(7):1129--1214, 2014.

\bibitem{CohnElkiesPropp}
H.~Cohn, N.~Elkies, and J.~Propp.
\newblock Local statistics for random domino tilings of the {A}ztec diamond.
\newblock {\em Duke Math. J.}, 85(1):117--166, 1996.

\bibitem{CometsGantertZeitouni}
F.~Comets, N.~Gantert, and O.~Zeitouni.
\newblock Quenched, annealed and functional large deviations for
  one-dimensional random walk in random environment.
\newblock {\em Probab. Theory Relat. Fields}, 118:65--114, 2000.

\bibitem{Corwin}
I.~Corwin.
\newblock The {K}ardar-{P}arisi-{Z}hang equation and universality class.
\newblock {\em Random matrices: theory and applications}, 1, 2011.

\bibitem{CSS14}
I.~Corwin, T.~Sepp{\"a}l{\"a}inen, and H.~Shen.
\newblock The strict-weak lattice polymer.
\newblock {\em J. Stat. Phys.}, 160(4):1027--1053, 2015.

\bibitem{DemboZeitouni}
A.~Dembo and O.~Zeitouni.
\newblock {\em Large Deviations Techniques and Applications}.
\newblock Springer, second edition, 1998.

\bibitem{Emrah}
E.~Emrah.
\newblock Limit shapes for inhomogeneous corner growth models with exponential
  and geometric weights.
\newblock {\em Electron. Comm. Probab.}, 21(42):1--16, 2016.

\bibitem{GeorgSepp}
N.~Georgiou and T.~Sepp{\"a}l{\"a}inen.
\newblock Large deviation rate functions for the partition function in a
  log-gamma distributed random potential.
\newblock {\em Ann. Probab.}, 41(6):4248--4286, 2013.

\bibitem{GTW01}
J.~Gravner, C.~A. Tracy, and H.~Widom.
\newblock Limit theorems for height fluctuations in a class of discrete space
  and time growth models.
\newblock {\em J. Statist. Phys.}, 102(5-6):1085--1132, 2001.

\bibitem{GTW02b}
J.~Gravner, C.~A. Tracy, and H.~Widom.
\newblock Fluctuations in the composite regime of a disordered growth model.
\newblock {\em Comm. Math. Phys.}, 229(3):433--458, 2002.

\bibitem{GTW}
J.~Gravner, C.~A. Tracy, and H.~Widom.
\newblock A growth model in a random environment.
\newblock {\em Ann. Probab.}, 30(3):1340--1368, 2002.

\bibitem{Janjigian}
C.~Janjigian.
\newblock Large deviations of the free energy in the {O}'{C}onnell-{Y}or
  polymer.
\newblock {\em J. Stat. Phys.}, 160(4):1054--1080, 2015.

\bibitem{Johansson}
K.~Johansson.
\newblock Shape fluctuations and random matrices.
\newblock {\em Communications in Mathematical Physics}, pages 437--476, 2000.

\bibitem{JockuschProppShor}
W.~Josckusch, J.~Propp, and P.~Shor.
\newblock Random domino tilings and the arctic circle theorem.
\newblock 1998, ar{X}iv:math/9801068v1.

\bibitem{Kesten}
H.~Kesten.
\newblock Aspects of first passage percolation.
\newblock In {\em \'{E}cole d'\'et\'e de probabilit\'es de {S}aint-{F}lour,
  {XIV}---1984}, volume 1180 of {\em Lecture Notes in Math.}, pages 125--264.
  Springer, Berlin, 1986.

\bibitem{Liggett}
T.~M. Liggett.
\newblock An improved subadditive ergodic theorem.
\newblock {\em Ann. Probab.}, 13(4):1279--1285, 1985.

\bibitem{Martin}
J.~B. Martin.
\newblock Last-passage percolation with general weight distribution.
\newblock \url{www.stats.ox.ac.uk/~martin/papers/lppsurvey2.ps}, 2005.

\bibitem{Rassoul-AghaSeppalainen}
F.~Rassoul-Agha and T.~Sepp{\"a}l{\"a}inen.
\newblock {\em A Course on Large Deviations with an Introduction to Gibbs
  Measures}, volume 162 of {\em Graduate Studes in Mathematics}.
\newblock American Mathematics Society, 2015.

\bibitem{Rockafellar}
R.~T. Rockafellar.
\newblock {\em Convex Analysis}, volume~28 of {\em Princeton Mathematical
  Series}.
\newblock Princeton University Press, 1970.

\bibitem{Rost}
H.~Rost.
\newblock Nonequilibrium behavior of a many particle process: Density profile
  and local equilibria.
\newblock {\em Z. Wahrsch. Verw. Gebiete}, 58:41--53, 1981.

\bibitem{Seppalainen3}
T.~Sepp{\"a}l{\"a}inen.
\newblock Coupling the totally asymmetric simple exclusion process with a
  moving interface.
\newblock {\em Markov Process. Relat. Fields}, 4(4):593--628, 1998.

\bibitem{Seppalainen2}
T.~Sepp{\"a}l{\"a}inen.
\newblock Hydrodynamic scaling, convex duality and asymptotic shapes of growth
  models.
\newblock {\em Markov Process. Relat. Fields}, 4(1):1--26, 1998.

\bibitem{Seppalainen4}
T.~Sepp{\"a}l{\"a}inen.
\newblock Large deviations for increasing sequences on the plane.
\newblock {\em Probab. Theory Relat. Fields}, 112(2):221--244, 1998.

\bibitem{Seppalainen}
T.~Sepp{\"a}l{\"a}inen.
\newblock Lecture notes on the corner growth model.
\newblock \url{http://www.math.wisc.edu/~seppalai/cornergrowth-book/ajo.pdf},
  2009.

\bibitem{Sep12}
T.~Sepp{\"a}l{\"a}inen.
\newblock Scaling for a one-dimensional directed polymer with boundary
  conditions.
\newblock {\em Ann. Probab.}, 40(1):19--73, 2012.

\bibitem{KrugSepp}
T.~Sepp{\"a}l{\"a}inen and J.~Krug.
\newblock Hydrodynamics and platoon formation for a totally asymmetric
  exclusion model with particlewise disorder.
\newblock {\em J. Stat. Phys.}, 95:525--567, 1999.

\bibitem{Si58}
M.~Sion.
\newblock On general minimax theorems.
\newblock {\em Pacific J. Math.}, 8:171--176, 1958.

\bibitem{Spohn}
H.~Spohn.
\newblock K{PZ} scaling theory and the semidiscrete directed polymer model.
\newblock In {\em Random matrix theory, interacting particle systems, and
  integrable systems}, volume~65 of {\em Math. Sci. Res. Inst. Publ.}, pages
  483--493. Cambridge Univ. Press, New York, 2014.

\end{thebibliography}

\end{document}